\documentclass[ssy,preprint]{imsart}

\usepackage[latin1]{inputenc}
\usepackage{amssymb,amsmath,amsthm,amsfonts,dsfont}
\usepackage{graphicx}
\usepackage{tikz}
\usetikzlibrary{decorations.pathreplacing}
\usepackage{tikz-cd}
\usepackage{hyperref}
\usepackage[nodayofweek]{datetime}

\startlocaldefs
\newdateformat{mydate}{\shortmonthname[\THEMONTH]{ }\twodigit{\THEDAY}, \THEYEAR}

\numberwithin{equation}{section}

\newtheorem{theorem}{Theorem}[section]
\newtheorem{lemma}[theorem]{Lemma}
\newtheorem{proposition}[theorem]{Proposition}
\newtheorem{claim}[theorem]{Claim}
\newtheorem{corollary}[theorem]{Corollary}
\theoremstyle{definition}
\newtheorem{definition}{Definition}[section]
\newtheorem{remark}{Remark}[section]

\endlocaldefs

\usepackage[normalem]{ulem}
\usepackage{color}

\begin{document}

\begin{frontmatter}
\title{Delay, memory, and messaging tradeoffs in distributed service systems\thanksref{T1}\thanksref{T2}}
\thankstext{T1}{Research supported in part by an MIT Jacobs Presidential Fellowship and the NSF grant CMMI-1234062.}
\thankstext{T2}{A preliminary version of this paper appeared in the Proceedings of the 2016 ACM Sigmetrics Conference \cite{sigmetrics16}.}

\begin{aug}
\author{\fnms{David} \snm{Gamarnik}\ead[label=e1]{gamarnik@mit.edu}},
\author{\fnms{John N.} \snm{Tsitsiklis}\ead[label=e2]{jnt@mit.edu}}
\and
\author{\fnms{Martin} \snm{Zubeldia}\ead[label=e3]{zubeldia@mit.edu}}


\affiliation{Massachusetts Institute of Technology}

\address{David Gamarnik\\
Sloan School of Management\\
Massachusetts Institute of Technology\\
Cambridge, MA, 02139, USA.\\
\printead{e1}}

\address{John N. Tsitsiklis\\
Martin Zubeldia\\
Laboratory for Information and Decision Systems\\
Massachusetts Institute of Technology\\
Cambridge, MA, 02139, USA.\\
\printead{e2}\\
\phantom{E-mail:\ }\printead*{e3}}

\end{aug}


\begin{abstract}
We consider the following distributed service model: jobs with unit mean, exponentially distributed, and independent processing times arrive as a Poisson process of rate $\lambda n$, with $0<\lambda<1$, and are immediately dispatched by a centralized dispatcher to one of $n$  First-In-First-Out queues associated with $n$ identical servers. The dispatcher is endowed with a finite memory, and with the ability to exchange messages with the servers.

We propose and study a resource-constrained ``pull-based" dispatching policy that involves two parameters: (i) the number of memory bits available at the dispatcher, and (ii) the average rate at which  servers communicate with the dispatcher. We establish (using a fluid limit approach) that the asymptotic, as $n\to\infty$, expected queueing delay is zero when either (i) the number of memory bits grows logarithmically with $n$ and the message rate grows superlinearly with $n$, or (ii) the number of memory bits grows superlogarithmically with $n$ and the message rate is at least $\lambda n$. Furthermore, when the number of memory bits grows only logarithmically with $n$ and the message rate is proportional to $n$, we obtain a closed-form expression for the (now positive) asymptotic delay.

Finally, we demonstrate an interesting phase transition in the resource-constrained regime where the asymptotic delay is non-zero. In particular, we show that for any given $\alpha>0$ (no matter how small), if our policy only uses a linear message rate $\alpha n$, the resulting asymptotic delay is upper bounded, uniformly over all $\lambda<1$; this is in sharp contrast to the delay obtained when no messages are used ($\alpha = 0$), which grows as $1/(1-\lambda)$ when $\lambda\uparrow 1$, or when the popular power-of-$d$-choices is used, in which the delay grows as $\log(1/(1-\lambda))$.

\end{abstract}

%

\end{frontmatter}

\setcounter{tocdepth}{2}
\tableofcontents

\section{Introduction}
This paper addresses the tradeoffs between performance (delay) and resources (local memory and communication overhead) in large-scale queueing systems. More specifically, we study such tradeoffs in the context of the supermarket model \cite{mitzenmacher}, which describes a system in which incoming jobs are to be dispatched to one of several queues associated with different servers (see Figure \ref{fig:basicSetting}).

 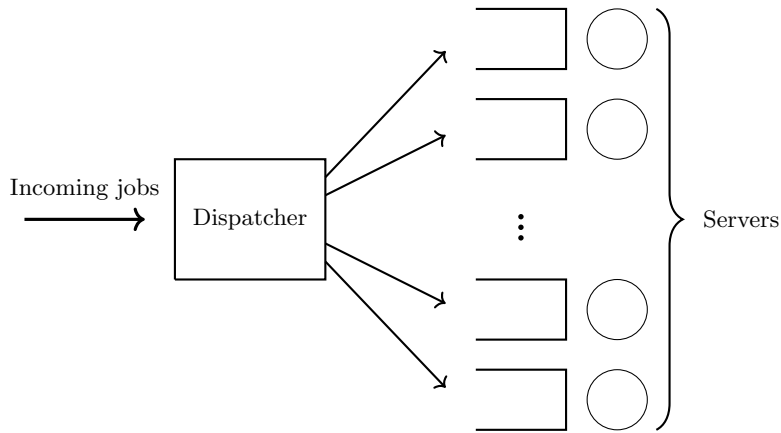
\begin{figure}[ht!]
 \centering
 \begin{tikzpicture}[scale=0.8]
    \draw [very thick,->] (-0.5,0) -> (1.5,0);
    \draw (0.5,0.5) node {Incoming jobs};
     \draw [thick] (2,-1) -- (2,1) -- (4.5,1) -- (4.5,-1) -- (2,-1);
     \draw (3.25,0) node {Dispatcher};
     \draw [thick, ->] (4.5,0.7) -- (6.5,2.8);
     \draw [thick, ->] (4.5,-0.7) -- (6.5,-2.8);
     \draw [thick, ->] (4.5,0.4) -- (6.5,1.4);
     \draw [thick, ->] (4.5,-0.4) -- (6.5,-1.4);
     \begin{scope}[yshift=3cm,xshift=5cm]
          \draw [thick] (2,0.5) -- (3.5,0.5) -- (3.5,-0.5) -- (2,-0.5);
          \draw (4.35,0) circle (5mm);
     \end{scope}
     \begin{scope}[yshift=1.5cm,xshift=5cm]
          \draw [thick] (2,0.5) -- (3.5,0.5) -- (3.5,-0.5) -- (2,-0.5);
          \draw (4.35,0) circle (5mm);
     \end{scope}
     \begin{scope}[xshift=5cm]
          \draw (2.75,0) node {\Huge\vdots};
     \end{scope}
     \begin{scope}[yshift=-3cm,xshift=5cm]
          \draw [thick] (2,0.5) -- (3.5,0.5) -- (3.5,-0.5) -- (2,-0.5);
          \draw (4.35,0) circle (5mm);
     \end{scope}
     \begin{scope}[yshift=-1.5cm,xshift=5cm]
          \draw [thick] (2,0.5) -- (3.5,0.5) -- (3.5,-0.5) -- (2,-0.5);
          \draw (4.35,0) circle (5mm);
     \end{scope}
          \draw [thick,decorate,decoration={brace,amplitude=10pt}]   (10,3.5) -- (10,-3.5) node [midway,right,xshift=.5cm] {Servers};
  \end{tikzpicture}
  \caption{The basic setting.}
    \label{fig:basicSetting}
  \end{figure}

There is a variety of ways that the system in the supermarket model can be operated, which correspond to different decision making architectures and policies, with different delay performance. At one extreme, incoming jobs can be sent to a random queue. This policy has no informational requirements but incurs a substantial delay because it does not take advantage of resource pooling. At the other extreme, incoming jobs can be sent to a shortest queue, or to a server with the smallest workload. The latter policies have very good performance (small queueing delay), but rely on substantial information exchange.

Many intermediate policies have been explored in the literature, and they achieve different performance levels while using varying amounts of resources, including local memory and communication overhead. For example, the power-of-$d$-choices \cite{mitzenmacher,vvedenskaya} and its variations \cite{shahFiniteLevels,srikant15,BorstPowerOfd} have been extensively studied, including the case of non-exponential service time distributions \cite{bramson13,kavita17}. More recently, pull-based policies like Join-Idle-Queue \cite{badonnelBurgess,joinIdleQueue} have been getting more attention, including extensions for heterogeneous servers \cite{stolyar14}, multiple dispatchers \cite{allerton16,infocom17,stolyar1}, and general service time distributions \cite{stolyar2}.

\subsection{Our contribution}
Our purpose is to study the effect of different resource levels (local memory and communication overhead), and to understand the amount of resources required for the asymptotic (as $n\to\infty$) delay to become negligible, in the context of the supermarket model. We adopt the average rate at which messages are exchanged between the dispatcher and the servers as our measure of the communication overhead, because of its simplicity and the fact that it applies to any kind of policy. We accomplish our purpose in two steps.
\begin{itemize}
\item[a)] In this paper, we propose a pull-based dispatching policy parameterized by the amount of resources involved, namely, the size of the memory used by the dispatcher and the average message rate.  We carry out a thorough analysis in different regimes and show that we obtain vanishing asymptotic delay if and only if the resources are above a certain level.
\item[b)] In a companion paper (see also \cite{sigmetrics16}), we show that in the regime (i.e., level of resources) where
our policy fails to result in vanishing asymptotic delay, the same is true for every other policy within a broad class of ``symmetric'' policies that  treat all servers in the same manner.
\end{itemize}

More concretely, our development relies on  a fluid limit approach. As is common with fluid-based analyses, we obtain two types of results: (i) qualitative results obtained through a deterministic analysis of a fluid model, and (ii) technical  results on the convergence of the actual  stochastic system to its fluid counterpart.

On the qualitative end, we establish the following:
 \begin{itemize}
\item [a)] If the message rate is superlinear in $n$ and the number of memory bits is at least logarithmic in $n$, then the asymptotic delay is zero.
\item [b)] If the message rate is at least $\lambda n$ and the number of memory bits is superlogarithmic in $n$, then the asymptotic delay is zero.
\item [c)] If the message rate is $\alpha n$ and the number of memory bits is $c\log_2(n)$, we derive a closed form expression for the (now positive) asymptotic delay, in terms of $\lambda$, $\alpha$, and $c$.
\item [d)] For the same amount of resources as in (c), we show an interesting phase transition in the asymptotic delay as the load approaches capacity ($\lambda\uparrow 1$). As long as a nontrivial linear message rate $\alpha n$, with $\alpha>0$,  is used, the asymptotic delay is uniformly upper bounded, over all $\lambda<1$. This is in sharp contrast to the delay obtained if no messages are used ($\alpha = 0$), which grows as $1/(1-\lambda)$ when $\lambda\uparrow 1$. This suggests that for large systems, even a small linear message rate provides significant improvements in the system's delay performance when $\lambda\uparrow 1$.
\item [e)] Again for the same amount of resources as in (c), we show a phase transition in the scaling of the asymptotic delay as a function of the memory parameter $c$, as we vary the message rate parameter $\alpha$.
     \begin{itemize}
     \item[(i)] If $\alpha<\lambda$, then the asymptotic delay is uniformly bounded away from zero, for any $c\geq 0$.
     \item[(ii)] If $\alpha=\lambda$, then the asymptotic delay decreases as $1/c$,  when $c\to\infty$.
     \item[(iii)] If $\alpha>\lambda$, then the queueing delay decreases as $(\lambda/\alpha)^c$, when $c\to\infty$.
     \end{itemize}
     This suggests that a message rate of at least $\lambda n$ is required for the memory to have a significant impact on the asymptotic delay.
\end{itemize}
On the technical end, and for each one of three regimes corresponding to cases (a), (b), and (c) above, we show the following:
\begin{itemize}
    \item [a)] The queue length process converges (as $n\to\infty$, and over any finite time interval) almost surely to the unique solution to a certain fluid model.
    \item [b)] For any initial conditions that correspond to starting with a finite average number of jobs per queue, the fluid solution converges (as time tends to $\infty$) to a unique invariant state.
    \item [c)] The steady-state distribution of the finite system converges (as $n\to\infty$) to the invariant state of the fluid model.
\end{itemize}

\subsection{Outline of the paper}
The rest of the paper is organized as follows. In Section \ref{sec:notation} we introduce some notation. In Section \ref{sec:results} we present the model and the main results, and also compare a particular regime of our policy to the so-called ``power-of-$d$-choices" policy. In Sections \ref{sec:proof_equilibrium}-\ref{sec:proof_inter} we provide the proofs of the main results. Finally, in Section \ref{sec:conclusions} we present our conclusions and suggestions for future work.

\section{Notation}\label{sec:notation}
In this section we introduce some notation that will be used throughout the paper. First, we define the notation for the asymptotic behavior of positive functions. In particular,
\begin{align*}
  f(n)\in o(g(n)) & \,\,\Leftrightarrow\,\, \limsup\limits_{n\to \infty} \frac{f(n)}{g(n)} = 0, \\
  f(n)\in O(g(n)) &\,\,\Leftrightarrow\,\, \limsup\limits_{n\to \infty} \frac{f(n)}{g(n)} < \infty, \\
  f(n)\in \Theta(g(n)) &\,\,\Leftrightarrow\,\, 0 < \liminf\limits_{n\to \infty} \frac{f(n)}{g(n)} \leq \limsup\limits_{n\to \infty} \frac{f(n)}{g(n)} < \infty, \\
  f(n)\in \Omega(g(n)) &\,\,\Leftrightarrow\,\, \liminf\limits_{n\to \infty} \frac{f(n)}{g(n)} > 0, \\
  f(n)\in \omega(g(n)) &\,\,\Leftrightarrow\,\, \liminf\limits_{n\to \infty} \frac{f(n)}{g(n)} = \infty.
\end{align*}
We let $[\,\cdot\,]^+\triangleq \max\{\,\cdot\,,0\}$, and denote by $\mathbb{Z}_+$ and $\mathbb{R}_+$ the sets of non-negative integers and real numbers, respectively. The indicator function is denoted by $\mathds{1}$, so that $\mathds{1}_A(x)$ is $1$ if $x\in A$, and is $0$ otherwise. The Dirac measure $\delta$ concentrated at a point $x$ is defined by $\delta_x(A)\triangleq\mathds{1}_A(x)$. We also define the following sets:
\begin{equation*}\label{eq:S}
 \mathcal{S}\triangleq\left\{s\in[0,1]^{\mathbb{Z_+}}: s_0=1; \,\, s_i \geq s_{i+1}, \,\, \forall \, i\geq 0 \right\},
\end{equation*}
\begin{equation}\label{eq:S1}
 \mathcal{S}^1\triangleq\left\{s\in\mathcal{S}: \sum\limits_{i=0}^\infty s_i < \infty  \right\},
\end{equation}
\[\mathcal{I}_n\triangleq\left\{ x\in[0,1]^{\mathbb{Z}_+}: x_i=\frac{k_i}{n}, \text{ for some } k_i\in\mathbb{Z}_+, \,\,\forall \, i \right\}. \]
We define the weighted $\ell_2$ norm $||\cdot||_w$ on $\mathbb{R}^{\mathbb{Z}_+}$ by
\[ ||x-y||_w^2\triangleq\sum\limits_{i=0}^{\infty} \frac{|x_i-y_i|^2}{2^i}. \]
Note that this norm comes from an inner product, so $(\ell^2_w,\|\cdot\|_w)$ is actually a Hilbert space, where
\[ \ell^2_w\triangleq \left\{ s\in\mathbb{R}^\mathbb{Z_+} : \|s\|_w<\infty \right\}. \]
We also define a partial order on $\mathcal{S}$ as follows:
\begin{align*}
  x \geq y \quad &\Leftrightarrow \quad x_i\geq y_i, \quad \forall \, i\geq 1, \\
  x > y \quad &\Leftrightarrow \quad x_i > y_i, \quad \forall \, i\geq 1.
\end{align*}
We will work with the Skorokhod spaces of functions
\[ D[0,T]\triangleq\left\{f:[0,T]\to\mathbb{R} : f \text{ is right-continuous with left limits} \right\}, \]
endowed with the uniform metric
\[ d(x,y)\triangleq\sup\limits_{t\in[0,T]} |x(t)-y(t)|, \]
and
\[ D^\infty[0,T]\triangleq \left\{f:[0,T]\to\mathbb{R}^{\mathbb{Z}_+} : f \text{ is right-continuous with left limits}\right\}, \]
with the metric
\[ d^{\mathbb{Z}_+}(x,y)\triangleq\sup\limits_{t\in[0,T]} ||x(t)-y(t)||_w. \]

\section{Model and main results}\label{sec:results}
In this section we present our main results. In Section \ref{sec:model_assumption} we describe the model and our assumptions. In Section \ref{sec:policyProperties} we introduce three different regimes of a certain pull-based dispatching policy. In Sections \ref{section:s+f} and \ref{section:tech} we introduce a fluid model and state the validity of fluid approximations for the transient and the steady-state regimes, respectively. In Section \ref{sec:comparison}, we discuss the asymptotic delay, and show a phase transition in its behavior when $\lambda\uparrow 1$.

\subsection{Modeling assumptions}\label{sec:model_assumption}
We consider a system consisting of $n$ parallel servers, where each server has a processing rate equal to $1$. Furthermore, each server is associated with an infinite capacity FIFO queue. We use the convention that a job that is being served remains in queue until its processing is completed. We assume that each server is work conserving: a server is idle if and only if the corresponding queue is empty.

Jobs arrive to the system as a single Poisson process of rate $\lambda n$ (for some fixed $\lambda<1$). Job sizes are i.i.d., independent from the arrival process, and exponentially distributed with mean $1$.

There is a central controller (dispatcher), responsible for routing each incoming job to a queue, immediately upon arrival. The dispatcher makes decisions based on  limited information about the state of the queues, as conveyed through messages from idle servers to the dispatcher, and which is stored in a limited local memory. See the next subsection for the precise description of the policy.

We will focus on the steady-state expectation of the time between the arrival of a typical job and the time at which it starts receiving service (to be referred to as ``{\bf queueing delay}'' or just ``{\bf delay}'' for short) and its limit as the system size $n$ tends to infinity (to be referred to as ``{\bf asymptotic delay}'').
Furthermore, we are interested in the amount of resources (memory size and message rate) required for the asymptotic delay to be equal to zero.

\subsection{Policy description and high-level overview of  the results} \label{sec:policyProperties}
In this section we introduce our policy and state in a succinct form our results for three of its regimes.

\subsubsection{Policy description}\label{sec:policy}
For any fixed value of $n$, the policy that we study operates as follows.
\begin{itemize}
\item [a)] {\bf Memory:} The dispatcher maintains a virtual queue comprised of up to $c(n)$ server identity numbers (IDs), also referred to as {\bf tokens}, so that the dispatcher's memory size is of order $c(n) \log_2(n)$ bits. Since there are only $n$ distinct servers, we will assume throughout the rest of the paper that $c(n)\leq n$.
\item [b)] {\bf Spontaneous messages from idle servers:} While a server is idle, it sends messages to the dispatcher as a Poisson process of rate $\mu(n)$, to inform or remind the dispatcher of its idleness. We assume that $\mu(n)$ is a nondecreasing function of $n$. Whenever the dispatcher receives a message, it adds the ID of the server that sent the message to the virtual queue of tokens, unless this ID is already stored or the virtual queue is full, in which cases the new message is discarded.
\item [c)] {\bf Dispatching rule:} Whenever a new job arrives, if there is at least one server ID in the virtual queue, the job is sent to the queue of a server whose ID is chosen uniformly at random from the virtual queue, and the corresponding token is deleted. If there are no tokens present, the job is sent to a queue chosen uniformly at random.
\end{itemize}
Note that under the above described policy, which is also depicted in Figure \ref{fig:policy},
no messages are ever sent from the dispatcher to the servers. Accordingly, following the terminology of \cite{badonnelBurgess}, we will refer to it as the Resource Constrained Pull-Based ({\bf RCPB}) policy or {\bf Pull-Based} policy for short.

\begin{figure}[ht!]
\begin{center}
 \begin{tikzpicture}
    \draw [very thick,->] (0.25,0) -> (1.25,0);
    \draw (0.75,0.5) node {$n\lambda$};
     \draw [thick] (1.5,-1) -- (1.5,1) -- (4.5,1) -- (4.5,-1) -- (1.5,-1);
     \draw (3,0) node {Dispatcher};
     \filldraw [thick,color=black!20] (3.5,2) -- (3.5,1.5) -- (2.5,1.5) -- (2.5,2) -- (3.5,2);
     \draw [thick] (3.5,2) -- (3.5,1.5) -- (2.5,1.5) -- (2.5,2) -- (3.5,2);
     \filldraw [thick,color=black!20] (3.5,2) -- (3.5,2.5) -- (2.5,2.5) -- (2.5,2) -- (3.5,2);
     \draw [thick] (3.5,2) -- (3.5,2.5) -- (2.5,2.5) -- (2.5,2) -- (3.5,2);
     \draw [thick] (3.5,3) -- (3.5,1.5) -- (2.5,1.5) -- (2.5,3);
     \draw [thick] (4,3) -- (2,3);
     \draw (4.5,3) node {$c(n)$};
     \draw (3,3.5) node {Queue of IDs};
     \draw [thick, ->] (4.5,0.7) -- (7.5,2.8);
     \draw (6.5,1.5) node {Jobs to};
     \draw (6.5,1) node {empty queues};
     \draw [thick, <-,dotted] (4.8,-0.7) -- (8,-3);
     \draw (5.5,-2.5) node {Messages from};
     \draw (5.5,-3) node {idle servers};
     \begin{scope}[yshift=3cm,xshift=6cm]
          \draw [thick] (2,0.5) -- (3.5,0.5) -- (3.5,-0.5) -- (2,-0.5);
          \draw (4.35,0) circle (5mm);
     \end{scope}
     \begin{scope}[yshift=1.5cm,xshift=6cm]
          \filldraw [thick,color=black!20] (3,0.5) -- (3.5,0.5) -- (3.5,-0.5) -- (3,-0.5) -- (3,0.5);
          \draw [thick] (3,0.5) -- (3.5,0.5) -- (3.5,-0.5) -- (3,-0.5) -- (3,0.5);
          \filldraw [thick,color=black!20] (3,0.5) -- (2.5,0.5) -- (2.5,-0.5) -- (3,-0.5) -- (3,0.5);
          \draw [thick] (3,0.5) -- (2.5,0.5) -- (2.5,-0.5) -- (3,-0.5) -- (3,0.5);
          \draw [thick] (2,0.5) -- (3.5,0.5) -- (3.5,-0.5) -- (2,-0.5);
          \draw (4.35,0) circle (5mm);
     \end{scope}
     \begin{scope}[xshift=6cm]
          \draw (2.75,0) node {\Huge\vdots};
     \end{scope}
     \begin{scope}[yshift=-3cm,xshift=6cm]
          \draw [thick] (2,0.5) -- (3.5,0.5) -- (3.5,-0.5) -- (2,-0.5);
          \draw (4.35,0) circle (5mm);
     \end{scope}
     \begin{scope}[yshift=-1.5cm,xshift=6cm]
          \draw [thick] (2,0.5) -- (3.5,0.5) -- (3.5,-0.5) -- (2,-0.5);
          \filldraw [thick,color=black!20] (3,0.5) -- (3.5,0.5) -- (3.5,-0.5) -- (3,-0.5) -- (3,0.5);
          \draw [thick] (3,0.5) -- (3.5,0.5) -- (3.5,-0.5) -- (3,-0.5) -- (3,0.5);
          \draw (4.35,0) circle (5mm);
     \end{scope}
          \draw [thick,decorate,decoration={brace,amplitude=10pt}]   (11,3.5) -- (11,-3.5) node [midway,right,xshift=.3cm] {$n$ servers};
  \end{tikzpicture}
  \end{center}
  \caption{Resource Constrained Pull-Based policy. Jobs are sent to queues associated with idle servers, based on tokens in the virtual queue. If no tokens are present, a queue is chosen at random.}
  \label{fig:policy}
\end{figure}
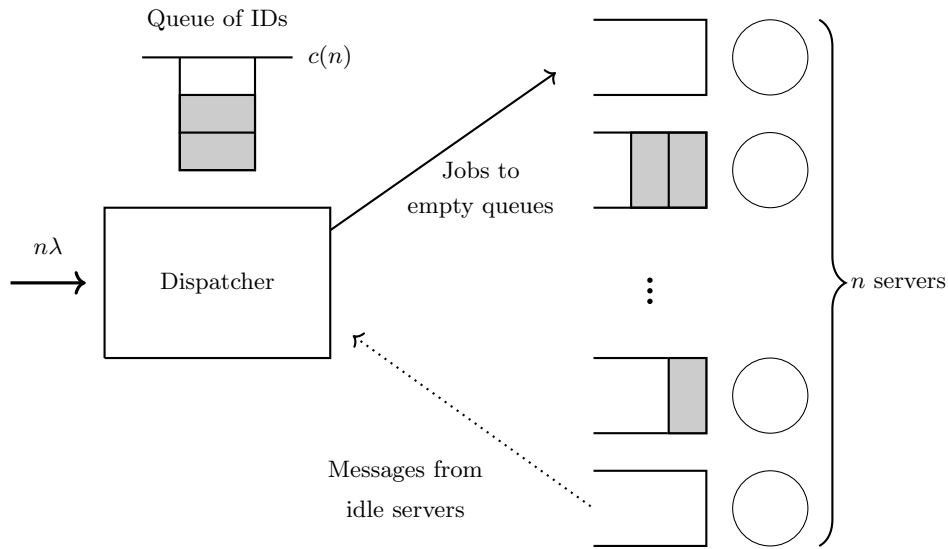
\subsubsection{High-level summary of the results}\label{sec:high-level}

We summarize our results for the RCPB policy, for three different regimes, in Table \ref{tab:regimes}, where we also introduce some mnemonic terms that we will use to refer to these regimes. Formal statements of these results are given later in this section. Furthermore, we provide a pictorial representation of the total resource requirements and the corresponding asymptotic delays in Figure \ref{fig:regimes}.

\begin{table}[ht!]
  \begin{center}
    \begin{tabular}{| c | c | c | c |}
    \hline
    Regime & Memory & Idle message rate & Delay \\ \hline\hline
    High Memory & $c(n)\in\omega(1)$ and $c(n)\in o(n)$ & $\mu(n)=\mu \geq\frac{\lambda}{1-\lambda}$ & $0$ \\ & & $\mu(n)=\mu<\frac{\lambda}{1-\lambda}$ & $>0$ \\\hline
    High Message & $c(n)=c \geq 1$ & $\mu(n)\in\omega(1)$ & $0$ \\ \hline
    Constrained & $c(n)=c \geq 1$ & $\mu(n)=\mu > 0$ & $>0$ \\ \hline
    \end{tabular}
\end{center}
\caption{The three regimes of our policy, and the resulting asymptotic delays.}
\label{tab:regimes}
\end{table}

\begin{figure}[ht!]
\begin{center}
\begin{tikzpicture}
  \filldraw[black!10] (4,3.5) -- (4,4) -- (4.5,4) -- (4.5,3.5) -- (4,3.5);
  \draw (4,3.5) -- (4,4) -- (4.5,4) -- (4.5,3.5) -- (4,3.5);
  \draw (5.25,3.75) node {Delay$>0$};
  \draw (4,2.9) -- (4,3.4) -- (4.5,3.4) -- (4.5,2.9) -- (4,2.9);
  \draw (5.25,3.15) node {Delay=0};

  \filldraw[black!10] (0,0) -- (0,2) -- (2.5,2) -- (2.5,0) -- (0,0);
  \filldraw[black!10] (0,0) -- (0,4) -- (1,4) -- (1,0) -- (0,0);

  \draw [->,thick] (-0.5,0) -- (5,0);
  \draw (6.5,0) node {Total message rate};
  \draw [->,thick] (0,-0.5) -- (0,4);
  \draw (0,4.3) node {Bits of memory};

  \draw (0.5,-0.3) node {$<\lambda n$};
  \draw (1.75,-0.3) node {$\Theta(n)$};
  \draw (3.75,-0.3) node {$\omega(n)$};

  \draw (-0.8,3) node {$\omega(\log(n))$};
  \draw (-0.8,1) node {$\Theta(\log(n))$};

  \draw (0,2) -- (5,2);
  \draw (2.5,0) -- (2.5,4);
  \draw [dotted,thick] (1,0) -- (1,4);

  \draw (1.25,3.15) node {High Memory};
  \draw (1.25,2.85) node {regime};

  \draw (3.75,1.15) node {High Message};
  \draw (3.75,0.85) node {regime};

  \draw (1.25,1.15) node {Constrained};
  \draw (1.25,0.85) node {regime};

\end{tikzpicture}
\end{center}
\caption{Resource requirements of the three regimes, and the resulting asymptotic delays.}
\label{fig:regimes}
\end{figure}
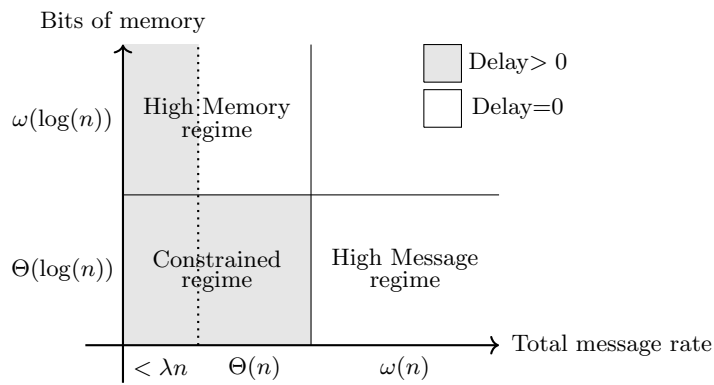

The more interesting subcase of the High Memory regime is when $\mu\geq \lambda/(1-\lambda)$, which results in zero asymptotic delay with superlogarithmic memory and linear overall message rate. Note that if we set $\mu=\lambda/(1-\lambda)$, and use the fact that servers are idle a fraction $1-\lambda$ of the time, the resulting time-average message rate becomes exactly $\lambda n$, i.e., one message per arrival.

\subsection{Stochastic and fluid descriptions of the system}
\label{section:s+f}
In this subsection, we define a stochastic process that corresponds to our model under the RCPB policy, as well as an associated fluid model.
\subsubsection{Stochastic system representation}\label{s:ssr}
Let $Q_i^n(t)$ be the number of jobs in queue $i$ (including the job currently being served, if any), at time $t$, in a $n$-server system.
We can model the system as a continuous-time Markov process whose state is the queue length vector, $Q^n(t)=\left(Q_i^n(t)\right)_{i=1}^n\in\mathbb{Z}_+^n$, together with the number of tokens, denoted by $M^n(t)\in\{0,1,\dots,c(n)\}$. However, as the system is symmetric with respect to the queues, we will use instead the more convenient representation $S^n(t)=\left(S^n_i(t)\right)_{i=0}^\infty$, where
\[ S^n_i(t)\triangleq\frac{1}{n} \sum\limits_{j=1}^{n} \mathds{1}_{[i,\infty)}\left(Q^n_j(t)\right), \quad  i \in \mathbb{Z}_+, \]
is the fraction of queues with at least $i$ jobs at time $t$.
Once more, the pair $\left(S^n(t), M^n(t)\right)$ is a continuous-time Markov process, with a countable state space.

Finally, another possible state representation involves $V^n(t)=\left(V^n_i(t)\right)_{i=1}^\infty$, where
\[ V^n_i(t)\triangleq\sum\limits_{j=i}^{\infty} S^n_j(t) \]
can be interpreted as the average amount by which a queue length exceeds $i-1$ at time $t$. In particular, $V_1^n(t)$ is the total number of jobs at time $t$ divided by $n$, and is finite, with probability $1$.

\subsubsection{Fluid model}

We now introduce the fluid model of $S^n(t)$, associated with our policy. Recall the definition of the set $\mathcal{S}^1$ in Equation \eqref{eq:S1}.

\begin{definition}[Fluid model]\label{def:fluid_model}
  Given an initial condition $s^0\in\mathcal{S}^1$, a continuous function $s(t):[0,\infty)\to\mathcal{S}^1$ is said to be a solution to the fluid model (or fluid solution) if:
  \begin{enumerate}
    \item $s(0)=s^0$.
    \item For all $t\geq 0$, $s_0(t)=1$.
    \item For all $t\geq 0$ outside of a set of Lebesgue measure zero, and for every $i\geq 1$, $s_i(t)$ is differentiable and satisfies
        \begin{align}
  \frac{ds_1}{dt}(t)=&\lambda\big(1-P_0(s(t))\big)+\lambda(1-s_1(t))P_0(s(t))-\big(s_1(t)-s_2(t)\big), \label{eq:drift1} \\
  \frac{ds_i}{dt}(t)=& \lambda\big(s_{i-1}(t)-s_i(t)\big)P_0(s(t))-\big(s_i(t)-s_{i+1}(t)\big) \quad \forall \, i\geq 2, \label{eq:drift2}
\end{align}
where $P_0(s)$ is given, for the three regimes considered, by:
\begin{itemize}
  \item [(i)] High Memory: $\quad P_0(s) = \left[ 1-\frac{\mu(1-s_1)}{\lambda} \right]^+$;
  \item [(ii)] High Message: $\quad P_0(s) = \left[ 1-\frac{1-s_2}{\lambda} \right]^+\mathds{1}_{\{s_1=1\}}$;
  \item [(iii)] Constrained: $\quad P_0(s)=\left[ \sum\limits_{k=0}^{c} \left( \frac{\mu(1-s_1)}{\lambda} \right)^k \right]^{-1}$.

      We use the convention $0^0=1$, so that the case $s_1=1$ yields $P_0(s)=1$.
\end{itemize}
  \end{enumerate}
\end{definition}


A solution to the fluid model, $s(t)$, can be thought of as a deterministic approximation to the sample paths of the stochastic process $S^n(t)$, for $n$ large enough. Note that the fluid model does not include a variable associated with the number of tokens. This is because, as we will see, the virtual queue process $M^n(t)$ evolves on a faster time scale than the processes of the queue lengths and does not have a deterministic limit. We thus have a process with two different time scales: on the one hand, the virtual queue evolves on a fast time scale (at least $n$ times faster) and from its perspective the queue process $S^n(t)$ appears static; on the other hand, the queue process $S^n(t)$ evolves on a slower time scale and from its perspective, the virtual queue appears to be at stochastic equilibrium. This latter property is manifested in the drift of the fluid model: $P_0(s(t))$ can be interpreted as the probability that the virtual queue is empty when the rest of the system is fixed at the state $s(t)$. Moreover, the drift of $s_1(t)$ is qualitatively different from the drift of the other components $s_i(t)$, for $i\geq 2$, because our policy treats empty queues differently.

We now provide some intuition for each of the drift terms in Equations \eqref{eq:drift1} and \eqref{eq:drift2}.
\begin{itemize}
  \item [(i)] $\lambda\big(1-P_0(s(t))\big)$: This term corresponds to arrivals to an empty queue while there are tokens in the virtual queue, taking into account that the virtual queue is nonempty with probability $1-P_0(s(t))$, in the limit.
  \item [(ii)] $\lambda\big(s_{i-1}(t)-s_i(t)\big)P_0(s(t))$: This term corresponds to arrivals to a queue with exactly $i-1$ jobs while there are no tokens in the virtual queue. This occurs when the virtual queue is empty and a queue with $i-1$ jobs is drawn, which happens with probability $P_0(s(t))\big(s_{i-1}(t)-s_i(t)\big)$.
  \item [(iii)] $-\big(s_i(t)-s_{i+1}(t)\big)$: This term corresponds to departures from queues with exactly $i$ jobs, which after dividing by $n$, occur at a rate equal to the fraction $s_i(t)-s_{i+1}(t)$ of servers with exactly $i$ jobs.
  \item [(iv)] Finally, the expressions for $P_0(s)$ are obtained through an explicit calculation of the steady-state distribution of $M^n(t)$ when $S^n(t)$ is fixed at the value $s$, while also letting $n\to\infty$.
\end{itemize}

Let us give an informal derivation of the different expressions for $P_0(s)$. Recall that $P_0(s)$ can be interpreted as the probability that the virtual queue is empty when the rest of the system is fixed at the state $s$. Under this interpretation, for any fixed state $s$, and for any fixed $n$, the virtual queue would behave like an M/M/1 queue with capacity $c(n)$, arrival rate $\mu(n)n(1-s_1)$, and departure rate $\lambda n$. In this M/M/1 queue, the steady-state probability of being empty is
\[ P_0^{(n)}(s)=\left[ \sum\limits_{k=0}^{c(n)} \left( \frac{\mu(n)(1-s_1)}{\lambda} \right)^k \right]^{-1}. \]
By taking the limit as $n\to\infty$, we obtain the correct expressions for $P_0(s)$, except in the case of the High Message regime with $s_1=1$. In that particular case, this simple interpretation does not work. However, we can go one step further and note that when all servers are busy (i.e., when $s_1=1$), servers become idle at rate $1-s_2$, which is the proportion of servers with exactly one job left in their queues. Since the high message rate assures that messages are sent almost immediately after the server becomes idle, only a fraction $[\lambda -(1 - s_2)]/\lambda$ of incoming jobs will go to a non-empty queue, which is exactly the probability of finding an empty virtual queue in this case.

\subsection{Technical results}\label{section:tech}
In this section we provide precise statements of our technical results.
\subsubsection{Properties of the fluid solutions}
The existence of fluid solutions will be established by showing that, almost surely, the limit of every convergent subsequence of sample paths of $S^n(t)$ is a fluid solution (Proposition \ref{prop:derivatives}). In addition,
the theorem that follows establishes uniqueness of fluid solutions for all initial conditions $s^0\in\mathcal{S}^1$, characterizes the unique equilibrium of the fluid model, and states its global asymptotic stability. The regimes mentioned in the statement of the results in this section correspond to the different assumptions on memory and message rates described in the 2nd and 3rd columns of Table \ref{tab:regimes}, respectively.

\begin{theorem}[Existence, uniqueness, and stability of fluid solutions] \label{thm:equilibrium}
A fluid solution, as described in Definition \ref{def:fluid_model}, exists and is unique for any initial condition $s^0\in\mathcal{S}^1$. Furthermore,  the fluid model has a unique equilibrium $s^*$, given by
\[ s_i^*=\lambda \left(\lambda P_0^*\right)^{i-1}, \quad \forall \, i\geq 1, \]
where $P_0^*=P_0(s^*)$ is given, for the three regimes considered, by:
\begin{itemize}
  \item [(i)] High Memory: $\quad P_0^* = \left[ 1-\frac{\mu(1-\lambda)}{\lambda} \right]^+$;
  \item [(ii)] High Message: $\quad P_0^* = 0$;
  \item [(iii)] Constrained: $\quad \label{eq:equilibrium_fixed_point}
        P_0^*=\left[ \sum\limits_{k=0}^{c} \left( \frac{\mu(1-\lambda)}{\lambda} \right)^k \right]^{-1}$.
\end{itemize}
This equilibrium is globally asymptotically stable, i.e.,
\[ \lim\limits_{t\to\infty} \left\|s(t)-s^*\right\|_w = 0, \]
for any initial condition $s^0\in\mathcal{S}^1$.
\end{theorem}
  The proof is given in Sections \ref{sec:proof_equilibrium} (uniqueness and stability) and \ref{sec:proof_fluid} (existence).

\begin{remark}
Note that, if $\mu\geq \lambda/(1-\lambda)$, the High Memory regime also has $P_0^*=0$ in equilibrium.
\end{remark}

\subsubsection{Approximation theorems}
The three results in this section justify the use of the fluid model as an approximation to the finite stochastic system. The first one states that the evolution of the process $S^n(t)$ is almost surely uniformly close, over any finite time horizon $[0,T]$, to the unique fluid solution $s(t)$.

\begin{theorem}[Convergence of sample paths]\label{thm:fluid_limit}
Fix $T>0$ and $s^0\in\mathcal{S}^1$. Under each of the three regimes, if
\[\lim\limits_{n\to \infty} \left\|S^n(0)-s^0\right\|_w=0, \quad\quad a.s., \]
then
\[\lim\limits_{n\to \infty} \sup\limits_{0\leq t \leq T} \left\|S^n(t)-s(t)\right\|_w=0, \quad\quad a.s., \]
where $s(t)$ is the unique fluid solution with initial condition $s^0$.
\end{theorem}
  The proof is given in Section \ref{sec:proof_fluid}.

\begin{remark}
 On the technical side, the proof is somewhat involved because the process $(S^n(t),M^n(t))$ is not the usual density-dependent Markov process studied by Kurtz \cite{densityKurtz} and which appears in the study of several dispatching policies (e.g., \cite{mitzenmacher,stolyar14,srikant15}). This is because $M^n(t)$ is not scaled by $n$, and consequently evolves in a faster time scale. We are dealing instead with an infinite-level infinite-dimensional jump Markov process, which is a natural generalization of its finite-level finite-dimensional counterpart studied in Chapter 8 of \cite{weiss}. The fact that our process may have infinitely many levels (memory states) and is infinite-dimensional prevents us from directly applying  known results. Furthermore, even if we truncated $S^n(t)$ to be finite-dimensional as in \cite{shahFiniteLevels}, our process still would not satisfy the more technical hypotheses of the corresponding result in \cite{weiss} (Theorem 8.15). Finally, the large deviations techniques used to prove Theorem 8.15 in \cite{weiss} do not directly generalize to infinite dimensions. For all of these reasons, we will prove our fluid limit result directly, by using a coupling approach, as in \cite{bramson98} and \cite{powerOfLittle}. Our results involve a separation of time scales similar to the ones in \cite{fadingMemories} and \cite{lossNetworks}.
\end{remark}

If we combine Theorems \ref{thm:fluid_limit} and \ref{thm:equilibrium}, we obtain that after some time, the state of the finite system $S^n(t)$ can be approximated by the equilibrium of the fluid model $s^*$, because
\[ S^n(t) \xrightarrow{n\to\infty} s(t) \xrightarrow{t\to\infty} s^*, \]
almost surely. If we interchange the order of the limits over $n$ and $t$, we obtain the limiting behavior of the invariant distribution $\pi^n_s$ of $S^n(t)$ as $n$ increases. In the next proposition and theorem, we show that the result is the same, i.e., that
\[ S^n(t) \xrightarrow{t\to\infty} \pi^n_s \xrightarrow{n\to\infty} s^*, \]
in distribution, so that the interchange of limits is justified.

The first step is to show that for every finite $n$, the stochastic process of interest is positive recurrent.

\begin{proposition}[Stochastic stability]\label{prop:ergodicity}
  For every $n$, the Markov process $\left(S^n(t),M^n(t)\right)$ is positive recurrent and therefore has a unique invariant distribution $\pi^n$.
\end{proposition}
  The proof is given in Section \ref{sec:proof_stoch_stab}.\\

Given $\pi^n$, the unique invariant distribution of the process $\left(S^n(t),M^n(t)\right)$, let
\[ \pi^n_s(\cdot) =\sum_{m=0}^{c(n)} \pi^n(\cdot,m)\]
be the marginal for $S^n$. We have the following result concerning the convergence of this sequence of marginal distributions.

\begin{theorem}[Convergence of invariant distributions] \label{prop:interchange}
  We have
  \[ \lim\limits_{n\to\infty} \pi^n_s = \delta_{s^*}, \quad\text{in distribution}. \]
\end{theorem}
  The proof is given in Section \ref{sec:proof_interchange}.\\

Putting everything together, we conclude that when $n$ is large, the fluid model is an accurate approximation to the stochastic system, for both the transient regime (Theorems \ref{thm:fluid_limit} and \ref{thm:equilibrium}) and the steady-state regime (Theorem \ref{prop:interchange}). The relationship between the convergence results is depicted in the commutative diagram of Figure \ref{fig:limits_diagram}.

\begin{figure}[ht!]
  \centering
  \begin{tikzpicture}[scale=1]
    \draw (0,0) node {$\pi^n_s$};
    \draw [<-,thick] (0,0.4) -- (0,2.6);
    \draw (1.5,0.3) node {Thm. \ref{prop:interchange}};
    \draw (1.5,-0.2) node {$n\to\infty$};
    \draw [->,thick] (0.4,0) -- (2.6,0);
    \draw (-0.8,1.7) node {Prop. \ref{prop:ergodicity}};
    \draw (-0.8,1.3) node {$t\to\infty$};
    \draw (0,3) node {$S^n(t)$};
    \draw [->,thick] (0.4,3) -- (2.6,3);
    \draw (3.75,1.7) node {Thm. \ref{thm:equilibrium}};
    \draw (3.75,1.3) node {$t\to\infty$};
    \draw (3,3) node {$s(t)$};
    \draw [->,thick] (3,2.6) -- (3,0.4);
    \draw (1.5,3.3) node {Thm. \ref{thm:fluid_limit}};
    \draw (1.5,2.8) node {$n\to\infty$};
    \draw (3,0) node {$s^*$};
  \end{tikzpicture}
  \caption{Relationship between the stochastic system and the fluid model.}
  \label{fig:limits_diagram}
\end{figure}

\subsection{Asymptotic  delay and phase transitions}\label{sec:comparison}

In this section we use the preceding results to conclude that in two of the regimes considered, the asymptotic delay is zero. For the third regime, the asymptotic delay is positive and we examine its dependence on various policy parameters.

\subsubsection{Queueing delay}

Having shown that we can approximate the stochastic system by its fluid model for large $n$, we can analyze the equilibrium of the latter to approximate the queueing delay under our policy.

For any given $n$, we define the  {\bf queueing delay} (more precisely, the waiting time) of a job,  generically denoted by $\mathbb{E}\left[W^n\right]$, as the mean time that a job spends in queue until its service starts. Here the expectation is taken with respect to the steady-state distribution, whose existence and uniqueness is guaranteed by Proposition \ref{prop:ergodicity}. Then, the {\bf asymptotic delay} is defined as
 \[ \mathbb{E}[W] \triangleq \limsup_{n\to\infty} \mathbb{E}\left[W^n\right]. \]
This asymptotic delay can be obtained from the equilibrium $s^*$ of the fluid model as follows. For a fixed $n$, the expected number of jobs in the system in steady-state is
\[ \mathbb{E}\left[  \sum\limits_{i=1}^{\infty} n S_i^n \right]. \]
Furthermore, the delay of a job is equal to the total time it spends in the system minus the expected service time (which is $1$). Using Little's Law, we obtain that the queueing delay is
\begin{align*}
  \mathbb{E}\left[ W^n \right] &= \frac{1}{\lambda n} \mathbb{E}\left[  \sum\limits_{i=1}^{\infty} n S_i^n \right] -1 = \frac{1}{\lambda}  \mathbb{E}\left[ \sum\limits_{i=1}^{\infty} S_i^n \right] -1.
\end{align*}
Taking the limit as $n\to\infty$, and interchanging the limit, summation, and expectation, we obtain
\begin{equation}\label{eq:delay_formula}
  \mathbb{E}\left[ W \right] = \frac{1}{\lambda} \left( \sum\limits_{i=1}^{\infty} s_i^* \right) -1.
\end{equation}
The validity of these interchanges is established in Appendix \ref{app:delay}.

As a corollary, we obtain that if we have a superlinear message rate or a superlogarithmic number of memory bits, the RCPB policy results in zero asymptotic delay.

\begin{corollary}
  For the High Memory regime with $\mu\geq\lambda/(1-\lambda)$, and for the High Message regime, the asymptotic delay is zero, i.e., $\mathbb{E}[W]=0$.
\end{corollary}
\begin{proof}
 From Theorem \ref{thm:equilibrium}, we have $P_0^*=0$ and therefore, $s_1^*=\lambda$ and $s_i^*=0$, for $i\geq 2$. The result follows from Equation \eqref{eq:delay_formula}.
\end{proof}

\subsubsection{The asymptotic delay in the Constrained regime} $ $
According to Equation \eqref{eq:delay_formula} and Theorem \ref{thm:equilibrium}, the asymptotic delay is given by
\begin{equation}
 \mathbb{E}[W]=\frac{1}{\lambda}\sum\limits_{i=1}^{\infty} s_i^* -1= \sum\limits_{i=1}^{\infty} (\lambda P_0^*)^{i-1} -1 = \frac{\lambda P_0^*}{1-\lambda P_0^*}, \label{eq:delay0}
 \end{equation}
and is positive in the Constrained regime. Nevertheless, the dependence of the delay on the various parameters has some remarkable properties, which we proceed to study.

Suppose  that the message rate of each idle server is $\mu=\alpha/(1-\lambda)$ for some constant $\alpha> 0$. Since a server is idle (on average) a fraction $1-\lambda$ of the time, the resulting average message rate at each server is $\alpha$, and the overall (system-wide) average message rate is $\alpha n$. We can rewrite the equilibrium probability $P^*_0$ in Theorem \ref{thm:equilibrium} as
\[ P_0^*=\left[ 1 + \frac{\alpha}{\lambda} + \cdots +  \left( \frac{\alpha}{\lambda}\right)^c \right]^{-1}. \]
This, together with Equation \eqref{eq:delay0} and some algebra, implies that
\begin{equation}\label{eq:delay}
  \mathbb{E}[W] = \lambda \left[ 1 - \lambda + \frac{\alpha}{\lambda} + \cdots +  \left( \frac{\alpha}{\lambda}\right)^c \right]^{-1}.
\end{equation}

\paragraph{Phase transition of the delay for $\lambda\uparrow 1$}
We have a phase transition between $\alpha=0$ (which corresponds to uniform random routing) and $\alpha>0$. In the first case, we have the usual M/M/1-queue delay: $\lambda/(1-\lambda)$. However, when $\alpha>0$, the delay is upper bounded uniformly in $\lambda$ as follows:
\begin{equation}
  \mathbb{E}[W]\leq \left( \sum\limits_{k=1}^{c} \alpha^k \right)^{-1}.
\end{equation}
This is established by noting that the expression in Equation
\eqref{eq:delay} is monotonically increasing in $\lambda$ and then setting $\lambda=1$.
Note that when $\alpha$ is fixed, the total message rate is the same, $\alpha n$, for all $\lambda<1$. This is a key qualitative improvement over all other resource constrained policies in the literature; see our discussion of the power-of-$d$-choices policy at the end of this subsection.

\paragraph{Phase transition in the memory-delay tradeoff}
When $\lambda$ and $\alpha$ are held fixed, the asymptotic delay in Equation \eqref{eq:delay} decreases with $c$. This represents a tradeoff between the asymptotic delay $\mathbb{E}[W]$, and the number of memory bits, which is equal to $\lceil c\log_2(n)\rceil$ for the Constrained regime. However, the rate at which the delay decreases with $c$ depends critically on the value of $\alpha$, and we have a phase transition when $\alpha=\lambda$.
\begin{itemize}
  \item [(i)] If $\alpha<\lambda$, then
  \[ \lim\limits_{c\to\infty} \mathbb{E}[W]=\frac{\lambda(\lambda-\alpha)}{(1-\lambda)(\lambda-\alpha)+1}. \]
Consequently, if $\alpha<\lambda$, it is impossible to drive the delay to $0$ by increasing the value of $c$, i.e., by increasing the amount of memory available.
  \item [(ii)] If $\alpha=\lambda$, we have
  \[ \mathbb{E}[W]=\frac{1}{1-\lambda + c} \leq \frac{1}{c},\]
  and thus the delay converges to $0$ at the rate of $1/c$, as $c\to\infty$.
  \item [(iii)] If $\alpha>\lambda$, we have
\begin{equation}   \label{eq:expo}
 \mathbb{E}[W]=\lambda \left[ 1 - \lambda + \frac{\alpha}{\lambda} + \cdots +  \left( \frac{\alpha}{\lambda}\right)^c \right]^{-1} \leq \left( \frac{\lambda}{\alpha} \right)^{c},
\end{equation}
   and thus the delay converges exponentially fast to $0$, as $c\to\infty$.
\end{itemize}
This phase transition is due to the fact that the queueing delay depends critically on $P_0^*$, the probability that there are no tokens left in the dispatcher's virtual queue. In equilibrium, the number of tokens in the virtual queue evolves as a birth-death process with birth rate $\alpha$, death rate $\lambda$, and maximum population $c$, and has an invariant distribution which is geometric with ratio $\alpha/\lambda$. As a result,
as soon as $\alpha$ becomes larger than $\lambda$, this birth-death process has an upward drift, and the probability of being at state 0 (no tokens present) decays exponentially with the size of its state space. This argument captures the essence of the phase transition  at $\mu=\lambda/(1-\lambda)$ for the High Memory regime.

\paragraph{Comparison with the power-of-$d$-choices}
The power-of-$d$-choices policy que\-ries $d$ random servers at the time of each arrival and sends the arriving job to the shortest of the queried queues. As such, it involves $2\lambda dn$ messages per unit time. For a fair comparison, we compare this policy to our RCPB policy with $\alpha=2\lambda d$, so that the two policies have the same average message rate.

The asymptotic delay for the power-of-$d$-choices policy was shown in \cite{mitzenmacher,vvedenskaya} to be
\[ \mathbb{E}[W_{\textrm{Pod}}] = \sum\limits_{i=1}^{\infty} \lambda^{\frac{d^i-d}{d-1}} -1 \geq \lambda^d. \]
Thus, the delay decreases at best exponentially with $d$, much like the delay decreases exponentially with $c$ in our scheme (cf.~Equation \eqref{eq:expo}). However, increasing $d$ increases the number of messages sent, unlike our policy where the average message rate remains fixed at $\alpha n$.

Furthermore, the asymptotic delay in the power-of-$d$-choices when $\lambda\uparrow 1$ is shown in \cite{mitzenmacher} to satisfy
\[ \lim_{\lambda\uparrow 1} \frac{\mathbb{E}[W_{\textrm{Pod}}]}{\log \left( \frac{1}{1-\lambda} \right)} = \frac{1}{\log d}. \]
For any fixed $d$, this is an exponential improvement over the delay of randomized routing, but the delay is still unbounded as $\lambda\uparrow 1$. In contrast, the delay of our scheme has a constant upper bound, independent of $\lambda$.

In conclusion, if we set $\alpha=2d\lambda$, so that our policy and the power-of-$d$ policy use the same number of messages per unit of time, our policy results in much better asymptotic delay, especially when $\lambda\uparrow 1$, even if $c$ is as small as $1$.

\paragraph{Numerical results}
We implemented three policies in Matlab: the power-of-$2$-choices \cite{mitzenmacher,vvedenskaya}, our RCPB policy, and the PULL policy \cite{stolyar14}. We evaluate the algorithms in a system with $500$ servers. In our algorithm we used $c=2$, and $\alpha=\lambda$, so it has the same average message rate as the PULL policy ($500\lambda$ messages per unit of time), which is $4$ times less than what the power-of-$2$-choices utilizes. In Figure \ref{fig:queueingDelay} we plot the delay as a function of $\log\left( 1/(1-\lambda) \right)$.

\begin{figure}[ht!]
  \centering
  \includegraphics[width=0.75\textwidth]{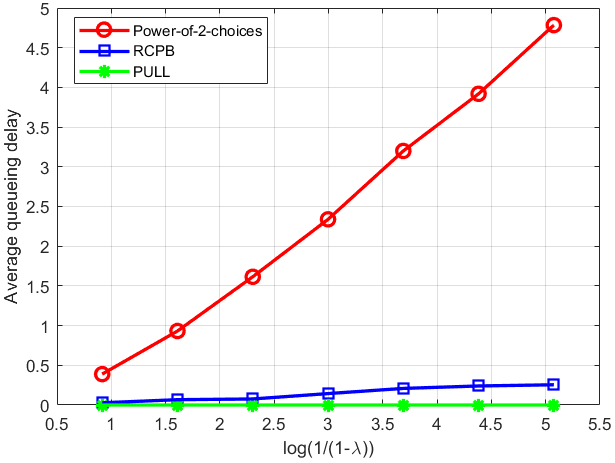}
        \caption{Average delay of the power-of-$2$-choices policy (red circles) vs. our policy (blue squares) vs. PULL (green asterisks).}
        \label{fig:queueingDelay}
\end{figure}

As expected, the delay remains uniformly bounded under our RCPB policy (blue squares). This is achieved with only $\lceil2\log_2(500)\rceil=18$ bits of memory. Furthermore, with this small amount of memory we are also close to the performance of the PULL algorithm, which requires $500$ bits of memory.

\section{Fluid model analysis --- Proof of part of Theorem \ref{thm:equilibrium}}\label{sec:proof_equilibrium}

The proof of Theorem \ref{thm:equilibrium} involves mostly deterministic arguments; these are developed in Lemmas \ref{lem:existence} and \ref{lem:equilibrium}, and Proposition \ref{prop:asymptitic_stability}, which establish uniqueness of fluid solutions, existence and uniqueness of a fluid-model equilibrium, and asymptotic stability, respectively. The proof of existence of fluid solutions relies on a stochastic argument and is developed in Section \ref{sec:proof_fluid}, in parallel with the proof of Theorem \ref{thm:fluid_limit}.

\subsection{Uniqueness of solutions}

\begin{lemma}\label{lem:existence}
If there exists a fluid solution (cf. Definition \ref{def:fluid_model}) with initial condition $s^0\in\mathcal{S}^1$, it is unique.
\end{lemma}
\begin{proof}
The fluid model is of the form $\dot{s}(t)=F\big(s(t)\big)$, where the function $F:\mathcal{S}^1\to[-1,\lambda]^{\mathbb{Z}_+}$ is defined by
  \begin{align}
     F_0(s)=&0, \nonumber \\
     F_1(s)=&\lambda\big(1-P_0(s)\big)+\lambda(1-s_1)P_0(s)-(s_1-s_2), \label{eq:driftS1} \\
     F_i(s)=& \lambda(s_{i-1}-s_i)P_0(s)-(s_i-s_{i+1}), \quad\quad \forall \, i\geq 2, \nonumber
\end{align}
and where $P_0(s)$ is given for the three regimes by:
\begin{itemize}
  \item [(i)] High Memory: $\quad P_0(s) = \left[ 1-\frac{\mu(1-s_1)}{\lambda} \right]^+$.
  \item [(ii)] High Message: $\quad P_0(s) = \left[ 1-\frac{1-s_2}{\lambda} \right]^+\mathds{1}_{\{s_1=1\}}$.
  \item [(iii)] Constrained: $\quad P_0(s)=\left[ \sum\limits_{k=0}^{c} \left( \frac{\mu(1-s_1)}{\lambda} \right)^k \right]^{-1}$.
\end{itemize}

The function $P_0(s)$ for the High Memory regime is continuous and piecewise linear in $s_1$, so it is Lipschitz continuous in $s$, over the set $\mathcal{S}^1$. Similarly, $P_0(s)$ for the Constrained regime is also Lipschitz continuous in $s$, because $P_0(s)$ is a rational function of $s_1$ and the denominator is lower bounded by $1$. However, $P_0(s)$ for the High Message regime is only Lipschitz continuous ``almost everywhere" in $\mathcal{S}^1$; more precisely, it is Lipschitz continuous everywhere except on the lower dimensional set
\[D\triangleq \left\{s\in\mathcal{S}^1:s_1=1 \text{ and } s_2>1-\lambda \right\}.\]
Moreover, $P_0(s)$ restricted to $D$ is also Lipschitz continuous.

Suppose that $P_0(s)$ is Lipschitz continuous with constant $L$ on some subset $\mathcal{S}_0$ of $\mathcal{S}^1$. Then, for every $s,s'\in\mathcal{S}_0$ and any $i\geq 1$, we have
  \begin{align*}
    \left|F_i(s)-F_i(s')\right| &= \left|-\lambda P_0(s)\mathds{1}_{i=1}+\lambda(s_{i-1}-s_{i})P_0(s)-(s_i-s_{i+1}) \right. \\
  & \left. \quad\quad\quad +\lambda P_0\left(s'\right)\mathds{1}_{i=1} -\lambda\left(s'_{i-1}-s'_i\right)P_0\left(s'\right) + \left(s'_i-s'_{i+1}\right)\right| \\
    &\leq \left|P_0(s)-P_0(s')\right| + \left|(s_{i-1}-s_i)P_0(s)-\left(s'_{i-1}-s'_i\right)P_0(s')\right| \\
& \quad\quad\quad + \left|s_i-s'_i\right| + \left|s_{i+1}-s'_{i+1}\right| \\
    &\leq 2\left|P_0(s)-P_0(s')\right| + \left|s_{i-1}-s'_{i-1}\right| + 2\left|s_i-s'_i\right| \\
& \quad\quad\quad+ \left|s_{i+1}-s'_{i+1}\right| \\
&\leq 2L \|s-s'\|_w + \left|s_{i-1}-s'_{i-1}\right| + 2\left|s_i-s'_i\right|+ \left|s_{i+1}-s'_{i+1}\right|.
  \end{align*}
Then,
  \begin{align*}
    &\left\|F(s)-F(s')\right\|_w= \sqrt{ \sum\limits_{i=0}^\infty \frac{\left|F_i(s)-F_i(s')\right|^2}{2^i} }
\end{align*}
\begin{align*}
&\leq \sqrt{ \sum\limits_{i=1}^\infty \frac{\Big(2L \|s-s'\|_w + \left|s_{i-1}-s'_{i-1}\right| + 2\left|s_i-s'_i\right|+ \left|s_{i+1}-s'_{i+1}\right|\Big)^2}{2^i} }\\
&\leq \sqrt{ 12 \sum\limits_{i=1}^\infty \frac{4L^2 \|s-s'\|_w^2 + \left|s_{i-1}-s'_{i-1}\right|^2 + 4\left|s_i-s'_i\right|^2+ \left|s_{i+1}-s'_{i+1}\right|^2}{2^i} } \\
    &\leq  \|s-s'\|_w  \sqrt{12( 4L^2 +2+4+1)},
  \end{align*}
where the second inequality comes from the fact that $(w+x+y+z)^2\leq 12(w^2+x^2+y^2+z^2)$, for all $(w,x,y,z)\in\mathbb{R}^4$. This means that $F$ is also Lipschitz continuous on the set $\mathcal{S}_0$.

For the High Memory and Constrained regimes, we can set $\mathcal{S}_0 = \mathcal{S}^1$, and by the preceding discussion, $F$ is Lipschitz continuous on $\mathcal{S}^1$. At this point we cannot immediately guarantee the uniqueness of solutions because $F$ is just Lipschitz continuous on a subset ($\mathcal{S}^1$) of the Hilbert space $(\ell^2_w,\|\cdot\|_w)$. However, we can use Kirszbraun's theorem \cite{LipschitzExtension} to extend $F$ to a Lipschitz continuous function $\overline{F}$ on the entire Hilbert space. If we have two different solutions to the equation $\dot{s}=F(s)$ which stay in $\mathcal{S}^1$, we would also have two different solutions to the equation $\dot{s}=\overline{F}(s)$. Since $\overline{F}$ is Lipschitz continuous, this would contradict the Picard-Lindel\"{o}ff uniqueness theorem \cite{picard}. This establishes the uniqueness of fluid solutions for the High Memory and Constrained regimes.

Note that the preceding argument can also be used to show uniqueness of solutions for any differential equation with a Lipschitz continuous drift in an arbitrary subset of the Hilbert space $(\ell^2_w,\|\cdot\|_w)$, as long as we only consider solutions that stay in that set. This fact will be used in the rest of the proof.

From now on, we concentrate on the High Message regime. In this case,
the drift $F(s)$ is Lipschitz continuous only ``almost everywhere," and a solution will in general be non-differentiable. In particular, results on the uniqueness of classical (differentiable) solutions do not apply. Our proof will rest on the fact that non-uniqueness issues can only arise when a trajectory hits the closure of the set where the drift $F(s)$ is not Lipschitz continuous, which in our case is just the closure of $D$:
\[ \overline{D} = \left\{s\in\mathcal{S}^1:s_1=1 \text{ and } s_2 \geq 1-\lambda \right\}. \]
We now partition the space $\mathcal{S}^1$ into three subsets, $\mathcal{S}^1\backslash \overline{D}$, $D$, and $\overline{D}\backslash D$, and characterize the behavior of potential trajectories depending on the initial condition. Note that we only consider fluid solutions, and these always stay in the set $\mathcal{S}^1$, by definition. Therefore, we only need to establish the uniqueness of solutions that stay in $\mathcal{S}^1$.

\begin{claim}\label{cla:trajectories}
For any fluid solution $s(t)$ in the High Message regime, and with initial condition $s^0\in\overline{D}$, we have the following.
\begin{itemize}
  \item [i)] If $s^0 \in D$, then $s(t)$ either stays in $D$ forever or hits $\overline{D} \backslash D$ at some finite time. In particular, it cannot go directly from $D$ to $\mathcal{S}^1\backslash \overline{D}$.
  \item [ii)] If $s^0 \in \overline{D}\backslash D$, then $s(t)$ stays in $\mathcal{S}^1\backslash D$ forever. In particular, it can never return to $D$.
\end{itemize}
\end{claim}
\begin{proof}$ $
\begin{itemize}
\item [i)] Suppose that $s^0\in D$, i.e., $s^0_1=1$ and $s^0_2>1-\lambda$. Let $t_{D^c}$ be the exit time from $D$, and suppose that it is finite. Note that, by continuity of solutions, $s_1(t_{D^c})=1$. We will show that $s_2(t_{D^c})=1-\lambda$, so that the trajectory hits $\overline{D} \backslash D$. Suppose, in order to derive a contradiction, that this is not the case and, therefore, $s_2(t_{D^c})>1-\lambda$. Then, due to the continuity of solutions, there exists some time $t_1>t_{D^c}$ such that $s_1(t_1)<1$ and $s_2(t)>1-\lambda$, for all $t\in[t_{D^c},t_1]$. Let
    \[ t_0\triangleq \sup\{t\leq t_1: s_1(t)=1\} \]
    be the last time before $t_1$ that $s_1(t)$ is equal to $1$. Then we have $s_1(t_0)=1$, and $s_1(t)<1$ for all $t\in(t_0,t_1]$. Since the drift $F$ is continuous for all $s_1<1$, all times in $(t_0,t_1]$ are regular. On the other hand, for all $t\in(t_0,t_1]$, we have $s_1(t)<1$ and thus $P_0(s(t))=0$, which together with $s_2(t)>1-\lambda$ implies that
\[ \frac{ds_1(t)}{dt}=\lambda -(s_1(t)-s_2(t)) > 0, \]
for all $t\in(t_0,t_1]$. This contradicts the relations $s_1(t_1)<1=s_1(t_0)$, and establishes that $s_1(t_D)=1$. Therefore the fluid solution $s$ either stays in $D$ forever or it exits $D$ with $s_2=1-\lambda$.
\item [ii)] Suppose now that $s^0\in \overline{D}\backslash D$, i.e., $s^0_1=1$ and $s^0_2=1-\lambda$. It is enough to show that $s_2(t)\leq 1-\lambda$, for all $t\geq 0$. Let
    \[ \tau_2(\epsilon) \triangleq \min \{ t\geq 0 : s_2(t)=1-\lambda+\epsilon \} \]
    be the first time $s_2$ reaches $1-\lambda+\epsilon$. Suppose, in order to derive a contradiction, that there exists $\epsilon^*>0$ such that $\tau_2(\epsilon^*)<\infty$. Then, due to the continuity of $s_2$, we also have $\tau_2(\epsilon)<\infty$, for all $\epsilon\leq \epsilon^*$. Since $s_2$ is differentiable almost everywhere, we can choose $\epsilon$ such that $\tau_2(\epsilon)$ is a regular time with $F_2(s(\tau_2(\epsilon)))>0$. Using the expression \eqref{eq:driftS1} for $F_2$, we obtain
\begin{align*}
 0 &< \lambda\Big(s_1(\tau_2(\epsilon))-s_2(\tau_2(\epsilon))\Big)\left( 1- \frac{1-s_2(\tau_2(\epsilon))}{\lambda} \right)\mathds{1}_{\{s_1(\tau_2(\epsilon))=1\}} \\
 & \qquad\qquad\qquad\qquad\qquad\qquad\qquad\qquad\qquad -\Big(s_2(\tau_2(\epsilon))-s_3(\tau_2(\epsilon))\Big) \\
 &\leq \lambda\Big(1-s_2(\tau_2(\epsilon))\Big)\left( 1- \frac{1-s_2(\tau_2(\epsilon))}{\lambda} \right)-\Big(s_2(\tau_2(\epsilon))-s_3(\tau_2(\epsilon))\Big) \\
 &= \lambda - 1 +s_3(\tau_2(\epsilon)) +s_2(\tau_2(\epsilon))\Big(1-\lambda-s_2(\tau_2(\epsilon))\Big) \\
 &< \lambda - 1 +s_3(\tau_2(\epsilon)),
\end{align*}
or $s_3(\tau_2(\epsilon))>1-\lambda$. On the other hand, we have $s_3(0)\leq s_2(0)=1-\lambda$. Combining these two facts, we obtain that $s_3(\tau_2(\epsilon))>s_3(0)$, i.e., that $s_3$ increased between times $0$ and $\tau_2(\epsilon)$. As a result, and since $s_3$ is differentiable almost everywhere, there exists another regular time $\tau_3(\epsilon)\leq \tau_2(\epsilon)$ such that $s_3(\tau_3(\epsilon))>1-\lambda$ and $F_3(s(\tau_3(\epsilon)))>0$. Proceeding inductively, we can obtain a sequence of nonincreasing regular times $\tau_2(\epsilon)\geq \tau_3(\epsilon)\geq \cdots \geq 0$ such that $s_k(\tau_k(\epsilon))>1-\lambda$, for all $k\geq 2$. Let $\tau_\infty(\epsilon)$ be the limit of this sequence of regular times. Since all coordinates of the fluid solutions are Lipschitz continuous with the same constant $L$, we have
\[ s_k(\tau_\infty)>1-\lambda-L(\tau_k(\epsilon)-\tau_\infty), \]
for all $k\geq 2$. Since $\tau_k(\epsilon)\to\tau_\infty$, there exists some $k^*\geq 2$ such that $s_k(\tau_\infty)>(1-\lambda)/2>0$, for all $k\geq k^*$. But then,
\[ \|s(\tau_\infty)\|_1 \geq \sum_{k=k^*}^\infty \frac{1-\lambda}{2} = \infty. \]
This contradicts the fact that $s(\tau_\infty)\in\mathcal{S}^1$, and it follows that we must have $s_2(t)\leq 1-\lambda$ for all $t\geq 0$.
\end{itemize}
\end{proof}

The uniqueness of a solution over the whole time interval $[0,\infty)$ for the High Message regime can now be obtained by concatenating up to three unique trajectories, depending on the initial condition $s^0$.
\begin{itemize}
\item [a)] Suppose that $s^0\in\mathcal{S}^1\backslash \overline{D}$, and let $t_{\overline{D}}$ be the hitting time of $\overline{D}$, i.e.,
\[ t_{\overline{D}} = \inf \left\{t\geq 0 : s(t)\in \overline{D} \text{ with } s(0)=s^0 \right\}. \]
Since $F|_{\mathcal{S}^1\backslash \overline{D}}$ (the restriction of the original drift $F$ to the set $\mathcal{S}^1\backslash \overline{D}$) is Lipschitz continuous, we have the uniqueness of a solution over the time interval $[0,t_{\overline{D}})$, by using the same argument as for the other regimes. If $t_{\overline{D}}=\infty$, then we are done. Otherwise, we have $s(t_{\overline{D}})\in \overline{D}$; the uniqueness of a solution over the time interval $[t_{\overline{D}},\infty)$ will immediately follow from the uniqueness of a solution with initial condition in $\overline{D}$.
\item [b)] Suppose that $s^0\in D$. Due to part i) of Claim \ref{cla:trajectories}, a solution can only exit the set $D$ by hitting $\overline{D}\backslash D$, and never by going back directly into $\mathcal{S}^1\backslash\overline{D}$. Let $t_{\overline{D}\backslash D}$ be the hitting time of $\overline{D}\backslash D$. Since $F|_D$ is Lipschitz continuous, we have uniqueness of a solution over the time interval $[0,t_{\overline{D}\backslash D})$. As in case a), if $t_{\overline{D}\backslash D}=\infty$ we are done. Otherwise, the uniqueness of a solution over the time interval $[t_{\overline{D}\backslash D},\infty)$ will immediately follow from the uniqueness of a solution with initial condition in $\overline{D}\backslash D$.
\item [c)] Suppose that $s^0\in \overline{D}\backslash D$. Due to part ii) of Claim \ref{cla:trajectories}, a solution stays in $\mathcal{S}^1\backslash D$ forever. As a result, since $F|_{\mathcal{S}^1\backslash D}$ is Lipschitz continuous, uniqueness follows.
\end{itemize}
\end{proof}

The intuition behind the preceding proof, for the High Message regime, is as follows. A non-differentiable solution may arise if the system starts with a large fraction of the servers having at least two jobs. In that case, the rate $s_1(t)-s_2(t)$ at which the servers become idle is smaller than the rate $\lambda$ at which idle servers become busy. As a consequence, the fraction $s_1(t)$ of busy servers increases until it possibly reaches its maximum of $1$, and stays there until the fraction of servers with exactly one job, which is now $1-s_2(t)$, exceeds the total arrival rate $\lambda$; after that time servers become idle at a rate faster than the arrival rate. This scenario is illustrated in Figure \ref{fig:discontinuity}.

\begin{figure}[ht!]
  \centering
  \includegraphics[width=0.7\textwidth]{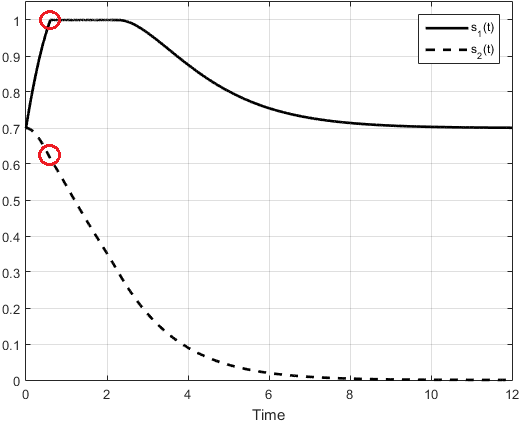}
  \caption{An example of a non-differentiable solution for the High Message regime, with $\lambda=0.9$, $s_1(0)=s_2(0)=s_3(0)=0.7$, and $s_i(0)=0$ for all $i\geq 4$. The solution is non-differentiable at the points indicated by the circles.}
  \label{fig:discontinuity}
\end{figure}

\subsection{Existence, uniqueness, and characterization of an equilibrium}

\begin{lemma} \label{lem:equilibrium}
The fluid model has a unique equilibrium $s^*\in\mathcal{S}^1$, given by
  \begin{align*}
  s_i^*=&\lambda (\lambda P_0^*)^{i-1}, \quad \forall \, i\geq 1,
\end{align*}
where $P_0^*\triangleq P_0(s^*)$ is given by
\begin{itemize}
  \item [(i)] High Memory: $P_0^* = \left[ 1-\frac{\mu(1-\lambda)}{\lambda} \right]^+$.
  \item [(ii)] High Message: $P_0^* = 0$.
  \item [(iii)] Constrained: $P_0^*=\left[ \sum\limits_{k=0}^{c} \left( \frac{\mu(1-\lambda)}{\lambda} \right)^k \right]^{-1}$.
\end{itemize}
\end{lemma}
\begin{proof}
A point $s^*\in\mathcal{S}^1$ is an equilibrium if and only if
\begin{align*}
  0=& \lambda\big(1-P_0(s^*)\big)+\lambda(1-s_1^*)P_0(s^*)-(s_1^*-s_2^*),  \\
  0=& \lambda(s_{i-1}^*-s_i^*)P_0(s^*)-(s_i^*-s_{i+1}^*), \quad\quad \forall \, i\geq 2.
\end{align*}
Since $s^*\in\mathcal{S}^1$, the sum $\sum_{i=0}^\infty (s_i^*-s_{i+1}^*)$ is absolutely convergent, even when we consider all the terms separately, i.e., when we consider $s_i^*$ and $-s_{i+1}^*$ as separate terms, for each $i\geq 0$. Thus, we can obtain equivalent equilibrium conditions by summing these equations over all coordinates $j\geq i$, for any fixed $i\geq 1$. We then obtain that $s^*$ is an equilibrium if and only if
\begin{align}
  0=& \lambda\big(1-P_0(s^*)\big)+\lambda P_0(s^*)\sum_{j=1}^\infty (s_{j-1}^*-s_j^*)-\sum_{j=1}^\infty(s_j^*-s_{j+1}^*), \label{eq:eq1} \\
  0=& \lambda P_0(s^*)\sum_{j=i}^\infty (s_{j-1}^*-s_j^*)-\sum_{j=i}^\infty(s_j^*-s_{j+1}^*), \quad\quad \forall \, i\geq 2. \label{eq:eq2}
\end{align}
Since the sums are absolutely convergent, we can rearrange the terms in Equations \eqref{eq:eq1} and \eqref{eq:eq2} to obtain that $s^*\in\mathcal{S}^1$ is an equilibrium if and only if
\begin{align*}
  0=&\ \lambda-s_1^*, \\
  0=& \ \lambda P_0(s^*)s_{i-1}^*-s_i^*, \quad \forall \, i\geq 2.
\end{align*}
These conditions yield $s_1^*=\lambda<1$, and
\begin{align*}
  s_i^*=&\lambda (\lambda P_0(s^*))^{i-1}, \quad \forall \, i\geq 1,
\end{align*}
which concludes the proof.
\end{proof}

\subsection{Asymptotic stability of the equilibrium}
We will establish global asymptotic stability by sandwiching a fluid solution between two solutions that converge to $s^*$, similar to the argument in \cite{vvedenskaya}. Towards this purpose, we first establish a monotonicity result.

\begin{lemma}\label{lem:monotonicity}
  Suppose that $s^1$ and $s^2$ are two fluid solutions with $s^1(0)\geq s^2(0)$. Then $s^1(t)\geq s^2(t)$, for all $t \geq 0$.
\end{lemma}
\begin{proof}
  It is known that uniqueness of solutions implies their continuous dependence on initial conditions, not only for the classical solutions in the High Memory and Constrained regimes, but also for the non-differentiable solutions of the High Message regime (see Chapter 8 of \cite{filippov}). Using this fact, it can be seen that it is enough to verify that $s^1(t)\geq s^2(t)$ when $s^1(0)> s^2(0)$, which we henceforth assume, under our particular definition of ``$>$" in Section \ref{sec:notation}. Let us define
\[ t_1=\inf \left\{ t\geq 0:s^1_k(t)< s^2_k(t),\text{ for some }k\geq 1 \right\}. \]
If $t_1=\infty$, then $s^1(t)\geq s^2(t)$ for all $t\geq 0$, and the result holds. It remains to consider the case where $t_1<\infty$, which we assume from now on.

By the definition of $t_1$, we have $s_{i}^1(t)\geq s_{i}^2(t)$ for all $i\geq 1$, and for all $t\leq t_1$. Since $P_0(s)$ is nondecreasing in $s$, this implies that $P_0(s^1(t))\geq P_0(s^2(t))$, for all $t\leq t_1$. Then, for all regular times $t\leq t_1$ and any $i\geq 2$, and also using the fact that $s_i$ is nonincreasing in $i$, we have
\begin{align*}
    F_i(s^1(t))-F_i(s^2(t))&= \lambda [s_{i-1}^1(t)-s_i^1(t)]P_0(s^1(t)) + [s_{i+1}^1(t)-s_{i+1}^2(t)] \\
 &\quad\quad\quad\quad - \lambda [s_{i-1}^2(t)-s_i^2(t)]P_0(s^2(t)) - [s_i^1(t)-s_i^2(t)] \\
&\geq \lambda [s_{i-1}^1(t)-s_i^1(t)]P_0(s^2(t))  \\
 &\quad\quad\quad\quad - \lambda [s_{i-1}^2(t)-s_i^2(t)]P_0(s^2(t)) - [s_i^1(t)-s_i^2(t)] \\
&\geq -\lambda P_0(s^2(t))[s_i^1(t)-s_i^2(t)] - [s_i^1(t)-s_i^2(t)] \\
&\geq -2 [s_i^1(t)-s_i^2(t)].
\end{align*}
Then, by Gr\"onwall's inequality we have
\begin{equation}\label{eq:gronwall}
 s_i^1(t)-s_i^2(t) \geq e^{-2t}[s_i^1(0)-s_i^2(0)], \quad \forall i\geq 2,
\end{equation}
for all $t\leq t_1$. This implies that $s_i^1(t)-s_i^2(t)>0$, for all $i\geq 2$ and for all $t\leq t_1$. It follows that, at time $t_1$, we must have $s_1^1(t_1)=s_1^2(t_1)$. The rest of the proof considers separately two different cases.

\noindent {\bf Case 1:} Suppose that we are dealing with the High Memory or the Constrained regime, or with the High Message regime with $s_1^1(t_1)=s_1^2(t_1)<1$. Since $s_1^1(t_1)=s_1^2(t_1)$, we have $P_0(s^1(t_1))=P_0(s^2(t_1))$. Then, due to the continuity of $s^1$, $s^2$, and of $P_0$ (local continuity for the High Message regime), there exists $\epsilon>0$ such that
\[ \lambda s_1^2(t) P_0(s^2(t)) - \lambda s_1^1(t)P_0(s^1(t)) - [s^1_1(t)-s^2_1(t)] > -\epsilon, \]
and (using Equation \eqref{eq:gronwall}) $s^1_2(t)-s^2_2(t)>\epsilon$, for all $t\leq t_1$ sufficiently close to $t_1$. As a result, we have
\begin{align}
    F_1(s^1(t))-F_1(s^2(t)) &= \lambda s_1^2(t) P_0(s^2(t)) - \lambda s_1^1(t)P_0(s^1(t))  \nonumber \\
&\quad\quad\quad  - [s^1_1(t)-s^2_1(t)] + [s^1_2(t)-s^2_2(t)] >0, \label{eq:driftDif1}
\end{align}
for all $t<t_1$ sufficiently close to $t_1$. Therefore, $s_1^1-s_1^2$ was increasing just before $t_1$. On the other hand, from the definition of $t_1$, we have $s_1^1(t_1)=s_1^2(t_1)$ and $s_1^1(t)\geq s_1^2(t)$ for all $t<t_1$. This is a contradiction, and therefore this case cannot arise.

\noindent {\bf Case 2:} Suppose now that we are dealing with the High Message regime, and that  $s_1^1(t_1)=s_1^2(t_1)=1$. Since $t_1<\infty$, we can pick a time $t_2>t_1$, arbitrarily close to $t_1$, such that $s_1^1(t_2)< s_1^2(t_2)$. Let us define
\[ t_1'=\sup \left\{ t\leq t_2:s^1_1(t)= s_1^2(t) \right\}. \]
Due to the continuity of $s^1$ and $s^2$, and since $s_1^1(t_1')=s_1^2(t_1')$ and $s_2^1(t_1)>s_2^2(t_1)$, there exists $\epsilon >0$ such that $s_1^2(t)-s_1^1(t)<\epsilon$ and $s_2^1(t)-s_2^2(t)>\epsilon$, for all $t\in[t_1',t_2]$ (we can always take a smaller $t_2$, if necessary, so that this holds). Furthermore, since $s_1^1(t)<1$ for all $t\in[t_1',t_2]$, we have $P_0(s^1(t))=0$, for all $t\in[t_1',t_2]$. Using these facts in Equation \eqref{eq:driftDif1}, we obtain $F_1(s^1(t))-F_1(s^2(t)) \geq 0$, for all $t\in[t_1',t_2]$. Therefore, $s_1^1-s_1^2$ is nondecreasing in that interval. This is a contradiction, because we have $s_1^1(t_1')=s_1^2(t_1')$ and $s_1^1(t_2)< s_1^2(t_2)$. Therefore, this case cannot arise either.
\end{proof}

We will now show that we can ``sandwich" any given trajectory $s(t)$ between a smaller one $s^l(t)$ and a larger one $s^u(t)$ (according to our partial order $\geq$) and prove that both $s^l(t)$ and $s^u(t)$ converge to $s^*$, to conclude that $s(t)$ converges to $s^*$.

\begin{proposition}\label{prop:asymptitic_stability}
The equilibrium $s^*$ of the fluid model is globally asymptotically stable, i.e.,
  \[ \lim\limits_{t\to\infty} \left\|s(t)-s^*\right\|_w = 0, \]
  for all fluid solutions $s(\cdot)$.
\end{proposition}
\begin{proof}
Suppose that $s(0)=s^0\in\mathcal{S}^1$. We define initial conditions $s^u(0)$ and $s^l(0)$ by letting
  \begin{align*}
    s_i^u(0) & =\max \left\{ s_i(0),s_i^* \right\}, \quad \text{and} \quad   s_i^l(0) =\min \left\{ s_i(0),s_i^* \right\},
  \end{align*}
for all $i$. We then have $s^u(0)\geq s^0 \geq s^l(0)$, $s^u(0)\geq s^* \geq s^l(0)$, and $s^u(0),s^l(0)\in\mathcal{S}^1$. Due to monotonicity (Lemma \ref{lem:monotonicity}), we obtain that $s^u(t)\geq s(t) \geq s^l(t)$ and $s^u(t)\geq s^* \geq s^l(t)$ for all $t\geq 0$. Thus it suffices to prove that $\left\|s^u(t)-s^*\right\|_w$ and $\left\|s^l(t)-s^*\right\|_w$ converge to $0$ as $t\to\infty$.

For any $s\in\mathcal{S}^1$, we introduce an equivalent representation in terms of a vector $v$ with components $v_i$ defined by
\begin{equation}\label{eq:Vdefinition}
v_i\triangleq \sum_{j=i}^\infty s_j,\qquad i\geq 1. \nonumber
\end{equation}
Note that any $s\in\mathcal{S}^1$ can be fully recovered from $v$. Therefore, we can work with a representation $v^u(t)$, $v^l(t)$, and $v^*$, of the vectors $s^u(t)$,  $s^l(t)$, and $s^*$, respectively.

From the proof of Lemma \ref{lem:existence}, we know that a trajectory can be non-differentiable at most at a single point in time. This can occur only for the High Message regime, and only if the trajectory hits the set
\[ D= \left\{s\in\mathcal{S}^1:s_1=1 \text{ and } s_2>1-\lambda \right\}, \]
where the drift is discontinuous. In all other cases, the trajectories are not only differentiable, but also Lipschitz continuous (in time), with the same Lipschitz constant for all coordinates. Therefore, in order to prove the asymptotic stability of the solutions, which is a property of the limiting behavior as $t\to\infty$, we can assume that the trajectories are everywhere differentiable and Lipschitz continuous.

Our first step is to derive a differential equation for $v_i$. This requires the interchange of summation and differentiation, which we proceed to justify. For any $i\geq 1$, we define a sequence of functions $\left\{f_k^{(i)}\right\}_{k=1}^\infty$, as follows:
\[ f_k^{(i)}(t) \triangleq \sum\limits_{j=i}^k \frac{ds_j^u}{dt}(t). \]
Using Equations \eqref{eq:drift1} and \eqref{eq:drift2}, we obtain
\begin{align*}
     f_k^{(1)}(t) &= \lambda - s_1^u(t)  +[s_{n+1}^u(t)-\lambda s_k^u(t) P_0(s^u(t))], \\
     f_k^{(i)}(t) &= \lambda s_{i-1}^u(t) P_0(s^u(t)) - s_i^u(t)  +[s_{k+1}^u(t)-\lambda s_k^u(t) P_0(s^u(t))], \quad \forall\, i\geq 2.
  \end{align*}
Since $s^u(t)\in\mathcal{S}^1$, for all $t$, we have the pointwise limits
  \begin{align*}
     \lim_{k\to\infty} f_k^{(1)}(t) &= \lambda - s_1^u(t), \\
     \lim_{k\to\infty}  f_k^{(i)}(t) &= \lambda s_{i-1}^u(t) P_0(s^u(t)) - s_i^u(t), \quad \forall\, i\geq 2.
  \end{align*}
On the other hand, since all components of $s^u(\cdot)$ are Lipschitz continuous with the same constant, and since $P_0(s)$ is also Lipschitz-continuous, the functions in the sequence $\left\{f_k^{(i)}\right\}_{k=1}^\infty$ are equicontinuous, for any given $i$. Then, the Arzel\`a-Ascoli theorem allows us to conclude that $f^{(i)}_k(\cdot)$ also converges uniformly, over any compact interval of time, to their pointwise limits. Using the uniform convergence, and the fact that $s^u(0)\in\mathcal{S}^1$, we can interchange summation and differentiation (Theorem 7.17 in \cite{rudin}) to obtain
\begin{align*}
     \frac{dv^u_1}{dt}(t) = \frac{d}{dt} \sum\limits_{j=1}^\infty s^u_j(t) = \sum\limits_{j=1}^\infty \frac{ds^u_j}{dt}(t) &= \lambda - s_1^u(t) \\
  \frac{dv^u_i}{dt}(t) = \frac{d}{dt} \sum\limits_{j=i}^\infty s^u_j(t) = \sum\limits_{j=i}^\infty \frac{ds^u_j}{dt} (t)   &= \lambda s_{i-1}^u(t)P_0(s^u(t))-s_i^u(t), \quad \forall\,i\geq 2.
\end{align*}
Turning the above differential equations into integral equations, and using the facts $s_1^*=\lambda$ and $\lambda s_{i-1}^*P_0^*-s_i^*=0$, we have
  \begin{align*}
    v^u_1(t)-v_1^u(0) =& \int\limits_0^t \left(s_1^* - s_1^u(\tau)\right) d\tau, \\
    v^u_i(t)-v^u_i(0) =& \int\limits_0^t \Big( \lambda \Big( s_{i-1}^u(\tau)P_0(s^u(\tau))- s_{i-1}^*P_0^*\Big)-\big(s_i^u(\tau)-s_i^*\big) \Big) d\tau.
  \end{align*}
Note that from the definition of $v_i$, we have $v_1^u(t)\geq v_i^u(t)$. Furthermore, from Lemma \ref{lem:monotonicity}, we have $s_1^u(t)\geq s_1^*$, so that $\dot{v}^u_1(t)\leq 0$, for all $t\geq 0$. It follows that
\[ v_1^u(0)\geq v_1^u(t)\geq v_i^u(t) \geq v^u_i(t)-v^u_i(0)\geq -v^u_i(0), \]
for all $t$.

We will now use induction on $i$ to prove  coordinate-wise convergence, i.e., that $\left| s^u_i(t) -s_i^* \right|$ converges to $0$ for all $i\geq 1$. We start with the base case, $i=1$. We have $s_1^u(\tau)-s_1^* \geq 0$, for all $\tau \geq 0$. Using the fact $\dot{v}_1^u(t)\leq 0$, we see that $v_1^u(t)$ converges to some limit, which we denote by $v_1^u(\infty)$. Then,
\[0 \leq \int\limits_0^\infty \left( s_1^u(\tau) - s_1^*\right) d\tau = v_1^u(0)-v_1^u(\infty) \leq v_1^u(0) < \infty, \]
which, together with the fact that $s_1$ is Lipschitz continuous, implies that $\left(s_1^u(\tau)- s_1^*\right)\to 0$ as $\tau \to \infty$.

We now consider some $i\geq 2$ and make the induction hypothesis that
\begin{equation}\label{eq:bounded1}
 \int\limits_0^\infty \left(s_k^u(\tau)-s_k^*\right) d\tau < \infty, \quad \forall \, k \leq i-1.
\end{equation}
Then,
\begin{equation}\label{eq:bounded2}
-v^u_i(0) \leq
v_i^u(t)-v_i^u(0)=
 \int\limits_0^t \Big( \lambda \Big( s_{i-1}^u(\tau)P_0(s^u(\tau))- s_{i-1}^*P_0^*\Big) - \big(s_i^u(\tau)-s_i^*\big) \Big) d\tau.
\end{equation}
  Adding and subtracting $\lambda s_{i-1}^*P_0(s^u(\tau))$ inside the integral, we obtain
  \begin{align}
     & -v^u_1(0)\leq \int\limits_0^t \Big( \lambda \left[ s_{i-1}^u(\tau) - s_{i-1}^*\right]P_0(s^u(\tau)) \nonumber  \\
      & \qquad\qquad\qquad\qquad\qquad  + \lambda\left[P_0(s^u(\tau)) - P_0^*\right]s_{i-1}^* -\big(s_i^u(\tau)-s_i^*\big) \Big) d\tau. \label{eq:before}
  \end{align}
Using Lemma \ref{lem:monotonicity}, we have $s_{i-1}^u(\tau)\geq s_{i-1}^*$ for all $i\geq 1$, and for all $\tau\geq 0$, which also implies that $P_0(s^u(\tau))\geq P_0^*$ for all $\tau\geq 0$. Therefore, the two terms inside brackets are nonnegative. Using the facts $\lambda< 1$, $s_{i-1}^*\leq 1$, and $P_0(s^u(\tau))\leq 1$, Equation \eqref{eq:before} implies that
\begin{equation*}
-v_i^u(0)\leq   \int\limits_0^t \Big( \left[ s_{i-1}^u(\tau) - s_{i-1}^*\right] + \left[P_0(s^u(\tau)) - P_0^*\right] -\left[s_i^u(\tau)-s_i^*\right] \Big) d\tau,
\end{equation*}
or
\begin{align}
  &\int\limits_0^t \left( s_i^u(\tau) - s_i^* \right) d\tau \nonumber \\
  &\qquad\qquad \leq v_i(0) + \int\limits_0^t \left( s_{i-1}^u(\tau) - s_{i-1}^* \right) d\tau
 + \int\limits_0^t \Big(P_0(s^u(\tau)) - P_0^*\Big) d\tau. \label{eq:stability_inequality}
\end{align}
The first integral on the right-hand side of Equation \eqref{eq:stability_inequality} is upper-bounded uniformly in $t$, by the induction hypothesis (Equation \eqref{eq:bounded1}). We now derive an upper bound on the last integral, for each one of the three regimes.
  \begin{itemize}
    \item [(i)] {\bf High Memory {regime}}:
By inspecting the expression for $P_0(s)$ for the High-Memory variant, we observe that it is  monotonically nondecreasing and Lipschitz continuous  in $s_1$.
Therefore, there exists a constant $L$ such that
 \[ \int\limits_0^t \big( P_0(s^u(\tau)) - P_0^* \big) d\tau
 \leq \int\limits_0^t L\big(s_1^u(\tau)-s_1^*\big) d\tau.\]
Using the induction hypothesis for $k=1$, we conclude that the last integral on the right-hand side of Equation \eqref{eq:stability_inequality} is upper bounded, uniformly in $t$.

    \item [(ii)] {\bf Constrained {regime}}:
    For the Constrained regime, the function $P_0(s)$ is again monotonically nondecreasing and, as remarked at the beginning of the proof of Lemma \ref{lem:existence}, it is also Lipschitz continuous in $s_1$. Thus, the argument is identical to the previous case.
    \item [(iii)] {\bf High Message regime}: We have an initial condition $s^0\in\mathcal{S}^1$, and therefore $0\leq v_1^0 < \infty$. As already remarked, we have $\dot{v}^u_1(t) = \lambda -s_1^u(t) \leq 0$. It follows that $s_1^u$ can be equal to $1$ for at most $v_1^0/(1-\lambda)$ units of time. Therefore, $P_0(s^u(t)) = \left[ 1-(1-s^u_2(t))/\lambda \right]^+\mathds{1}_{\{s^u_1(t)=1\}}$ can be positive only on a set of times of Lebesgue measure at most $v_1^0/(1-\lambda)$. This implies the uniform (in $t$) upper bound
        \[ \int\limits_0^t \Big( P_0(s^u(\tau)) - P_0^* \Big) d\tau = \int\limits_0^t P_0(s^u(\tau)) d\tau \leq \frac{v_1^0}{1-\lambda}. \]

    \end{itemize}

For all three cases, we have shown that the last integral in Equation \eqref{eq:stability_inequality} is upper bounded, uniformly in $t$. It follows from Equation \eqref{eq:stability_inequality} and the induction hypothesis that
 \[ \int\limits_0^\infty \left( s_i^u(\tau)-s_i^* \right) d\tau < \infty. \]
This completes the proof of the induction step. Using the Lipschitz-continuity of $s^u_i(\cdot)$, it follows that $s^u_i(t)$ converges to $s^*_i$ for all $i\geq 1$. It is straightforward to check that this coordinate-wise convergence, together with boundedness ($s^u_i(t)\leq 1$, for all $i$ and $t$), implies that also
   \[ \lim\limits_{t\to\infty} \|s^u(t)-s^*\|_w =0. \]
   An analogous argument gives us the convergence
   \[ \lim\limits_{t\to\infty} \|s^l(t)-s^*\|_w =0, \]
   which concludes the proof.
\end{proof}

\section{Stochastic transient analysis --- Proof of Theorem \ref{thm:fluid_limit} and of the rest of Theorem \ref{thm:equilibrium}}\label{sec:proof_fluid}
We will now prove the convergence of the stochastic system to the fluid solution.
The proof involves three steps. We first define the process using a coupled sample path approach, as in \cite{powerOfLittle}.  We then show the existence of limiting trajectories under the fluid scaling (Proposition \ref{prop:tightness}). We finally show that any such limit trajectory must satisfy the differential equations in the definition of the fluid model (Proposition \ref{prop:derivatives}).

\subsection{Probability space and coupling}

We will first define a common probability space for all $n$. We will then define a coupled sequence of processes $\left\{\left(S^n(t),M^n(t)\right)\right\}_{n=1}^\infty$. This approach will allow us to obtain almost sure convergence in the common probability space.

\subsubsection{Fundamental processes and initial conditions}

All processes of interest (for all $n$) will be driven by certain common fundamental processes.
\begin{itemize}
  \item [a)] Driving Poisson processes: Independent Poisson counting processes $\mathcal{N}_\lambda(t)$ (process of arrivals, with rate $\lambda$), and $\mathcal{N}_1(t)$ (process of potential departures, with rate $1$). A coupled sequence $\left\{\mathcal{N}_{\mu(n)}(t)\right\}_{n=1}^\infty$ (processes of potential messages, with nondecreasing rates $\mu(n)$), independent of $\mathcal{N}_\lambda(t)$ and $\mathcal{N}_1(t)$, such that the events in $\mathcal{N}_{\mu(n)}(t)$ are a subset of the events in $\mathcal{N}_{\mu(n+1)}(t)$ almost surely, for all $n\geq 1$. These processes are defined on a common probability space $(\Omega_D,\mathcal{A}_D,\mathbb{P}_D)$.
  \item [b)] Selection processes: Three independent discrete time processes $U(k)$, $V(k)$, and $W(k)$, which are all i.i.d. and uniform on $[0,1]$, defined on a common probability space $(\Omega_S,\mathcal{A}_S,\mathbb{P}_S)$.
  \item [c)] Initial conditions: A sequence of random variables $\left\{\left(S^{(0,n)},M^{(0,n)}\right)\right\}_{n=1}^\infty$ defined on a common probability space $(\Omega_0,\mathcal{A}_0,\mathbb{P}_0)$ and taking values in $\left(\mathcal{S}^1\cap\mathcal{I}_n\right)\times\{0,1,\dots,c(n)\}$.
\end{itemize}
The whole system will be defined on the probability space
\[ (\Omega,\mathcal{A},\mathbb{P})=(\Omega_D\times\Omega_S\times\Omega_0, \mathcal{A}_D\times\mathcal{A}_S\times\mathcal{A}_0, \mathbb{P}_D\times\mathbb{P}_S\times\mathbb{P}_0). \]
All of the randomness in the system (for any $n$) will be specified by these fundamental processes, and everything else will be a deterministic function of them.

\subsubsection{A coupled construction of sample paths}
Recall that our policy results in a Markov process $\left(S^n(t),M^n(t)\right)\in\left(\mathcal{S}^1 \cap\mathcal{I}_n\right)\times\{0,1,\dots,c(n)\}$, where $S^n_i(t)$ is the fraction of servers with at least $i$ jobs and $M^n(t)$ is the number of tokens stored in memory, at time $t$.  We now describe a particular construction of the process, as a deterministic function of the fundamental processes. We decompose the process $S^n(t)$ as the sum of two non-negative and non-decreasing processes, $A^n(t)$ and $D^n(t)$, that represent the (scaled by $n$) total cumulative arrivals to and departures from the queues, respectively, so that
\[ S^n(t)=S^{(0,n)}+A^n(t)-D^n(t). \]
Let $t^{\lambda,n}_j$, $t^{1,n}_j$, and $t^{\mu,n}_j$ be the time of the $j$-th arrival of $\mathcal{N}_\lambda(nt)$, $\mathcal{N}_1(nt)$, and $\mathcal{N}_{\mu(n)}(nt)$, respectively. In order to simplify notation, we will omit the superscripts $\lambda$, $1$, and $\mu$, when the corresponding process is clear. We denote by $S^n(t^-)$ the left limit $\lim\limits_{s\uparrow t} S^n(s)$, and similarly for $M^n(t^-)$. Then, the first component of $A^n(t)$ is
\begin{align}
A_1^n(t) = \frac{1}{n} \sum\limits_{j=1}^{\mathcal{N}_\lambda(nt)} & \Big[ \mathds{1}_{[1,c(n)]}\left(M^n\left({t_j^n}^-\right)\right) \nonumber \\
&  + \mathds{1}_{\{0\}} \left(M^n\left({t_j^n}^-\right)\right)\mathds{1}_{\left[0, 1-S_1^n\left({t_j^n}^-\right)\right)}(U(j)) \Big]. \label{eq:cum_arrivals}
\end{align}
The above expression is interpreted as follows. We have an upward jump of size $1/n$ in $A_1^n$ every time that a job joins an empty queue, which happens every time that there is an arrival and either (i) there are tokens in the virtual queue (i.e., $M^n>0$) or, (ii) there are no tokens and an empty queue is drawn uniformly at random, which happens with probability $1-S^n_1$. Similarly, for $i\geq 2$,
\[ A_i^n(t)=\frac{1}{n} \sum\limits_{j=1}^{\mathcal{N}_\lambda(nt)} \mathds{1}_{\{0\}}\left(M^n\left({t_j^n}^-\right)\right) \mathds{1}_{\left[1-S_{i-1}^n\left({t_j^n}^-\right),1-S_i^n\left({t_j^n}^-\right)\right)}(U(j)). \]
In this case we have an upward jump in $A_i^n$ of size $1/n$ every time that there is an arrival, there are no tokens in the virtual queue (i.e., $M^n=0$), and a queue with exactly $i-1$ jobs is drawn uniformly at random, which happens with probability $S^n_{i-1}-S^n_{i}$. Moreover, for all $i\geq 1$,
\[ D_i^n(t) = \frac{1}{n} \sum\limits_{j=1}^{\mathcal{N}_1(nt)} \mathds{1}_{\left[1-S_i^n\left({t_j^n}^-\right),1-S_{i+1}^n\left({t_j^n}^-\right)\right)} (W(j)). \]
We have an upward jump in $D_i^n$ of size $1/n$ when there is a departure from a queue with exactly $i$ jobs, which happens with rate $\left(S^n_{i}-S^n_{i+1}\right)n$.

Recall that $\mu(n)$ is the message rate of an empty server. In the High Memory and Constrained regimes, we have $\mu(n)=\mu$, while in the High Message regime $\mu(n)$ is a nondecreasing and unbounded sequence. Potential messages are generated according to the process $\mathcal{N}_{\mu(n)}(nt)$, but an actual message is generated only if a randomly selected queue is empty.
Thus, the number of tokens in the virtual queue evolves as follows:
\begin{align}
  &M^n(t) =  M^{(0,n)} - \sum\limits_{j=1}^{\mathcal{N}_\lambda(nt)} \mathds{1}_{[1,c(n)]}\left(M^n\left({t_j^n}^-\right)\right) \nonumber \\
  &+ \sum\limits_{j=1}^{\mathcal{N}_{\mu(n)}\left(nt\right)} \mathds{1}_{[0,c(n)-1]}\left(M^n\left({t_j^n}^-\right)\right) \mathds{1}_{\left[0,1-S_1\left({t_j^n}^-\right)-\frac{M^n\left({t_j^n}^-\right)}{n}\right)}(V(j)). \label{eq:memory_def}
\end{align}
To see this, if the virtual queue is not empty, a token is removed from the virtual queue each time there is an arrival. Furthermore, if the virtual queue is not full, a new token is added each time a new message arrives from one of the $n(1-S_1^n)-M^n$ queues that do not have corresponding tokens in the virtual queue.

\begin{remark}
The desired result only concerns the convergence of the projection of the Markov process $\left(S^n(t),M^n(t)\right)$ onto its first component. However, the process of tokens $M^n$ will still have an impact on that limit. 
\end{remark}

As mentioned earlier, the proof involves the following two steps:
\begin{enumerate}
  \item We show that there exists a measurable set $\mathcal{C}\subset\Omega$ with $\mathbb{P}(\mathcal{C})=1$ such that for all $\omega\in\mathcal{C}$, any sequence of sample paths $S^n(\omega,t)$ contains a further subsequence that converges to a Lipschitz continuous trajectory $s(t)$, as $n\to\infty$.
  \item We characterize the derivative of $s(t)$ at any regular point and show that it is identical to the drift of our fluid model. Hence $s(t)$ must be a fluid solution for some initial condition $s^0$, yielding also, as a corollary, the existence of fluid solutions.
\end{enumerate}

\subsection{Tightness of sample paths}

We start by finding a set of ``nice" sample paths $\omega$ for which any subsequence of the sequence $\left\{S^n(\omega,t)\right\}_{n=1}^\infty$ contains a further subsequence $\left\{S^{n_k}(\omega,t)\right\}_{k=1}^\infty$ that converges to some Lipschitz continuous function $s$. The arguments involved here are fairly straightforward and routine.

\begin{lemma}\label{lem:nice_set}
  Fix $T>0$. There exists a measurable set $\mathcal{C}\subset \Omega$ such that $\mathbb{P}(\mathcal{C})=1$ and for all $\omega\in\mathcal{C}$,
  \begin{align}
  &\lim\limits_{n\to\infty} \sup\limits_{t\in[0,T]} \left|\frac{1}{n}\mathcal{N}_\lambda(\omega,nt)-\lambda t\right|=0, \label{eq:FLLN1} \\
  &\lim\limits_{n\to\infty} \sup\limits_{t\in[0,T]} \left|\frac{1}{n}\mathcal{N}_1(\omega,nt)- t\right|=0, \label{eq:FLLN2} \\
  &\lim\limits_{n\to\infty} \frac{1}{n} \sum\limits_{i=1}^n \mathds{1}_{[a,b)}(U(\omega,i))=b-a, \quad \text{for all } [a,b)\subset[0,1], \label{eq:GC1}\\
  &\lim\limits_{n\to\infty} \frac{1}{n} \sum\limits_{i=1}^n \mathds{1}_{[c,d)}(W(\omega,i))=d-c, \quad \text{for all } [c,d)\subset[0,1]. \label{eq:GC2}
  \end{align}
\end{lemma}
\begin{proof}
  Using the Functional Strong Law of Large Numbers for Poisson processes, we obtain a subset $\mathcal{C}_D\subset\Omega_D$ such that $\mathbb{P}_D(\mathcal{C}_D)=1$ on which Equations \eqref{eq:FLLN1} and \eqref{eq:FLLN2} hold. Furthermore, the Glivenko-Cantelli lemma gives us another subset $\mathcal{C}_S\subset\Omega_S$ such that $\mathbb{P}_S(\mathcal{C}_S)=1$ and on which Equations \eqref{eq:GC1} and \eqref{eq:GC2} hold. Taking $\mathcal{C}=\mathcal{C}_D\times\mathcal{C}_S\times\Omega_0$ concludes the proof.
\end{proof}

Let us fix an arbitrary $s^0\in[0,1]$, sequences $R_n \downarrow 0$ and $\gamma_n\downarrow 0$, and a constant $L>0$. For $n\geq 1$, we define the following subsets of $D[0,T]$:
\begin{align}
  E_n(R_n,\gamma_n) \triangleq \Big\{ s\in D[0,T]:\,&|s(0)-s^0|\leq R_n, \text{ and } \nonumber \\
   &|s(a)-s(b)|\leq L|a-b| +\gamma_n, \,\,\forall \, a,b\in[0,T]\Big\}.
\end{align}
We also define
\[ E_c \triangleq \Big\{ s\in D[0,T]:\, s(0)=s^0,\,\, |s(a)-s(b)|\leq L|a-b|, \,\,\forall \, a,b\in[0,T]\Big\}, \]
which is the  set of $L$-Lipschitz continuous functions with fixed initial conditions, and which is known to be sequentially compact, by the
Arzel\`a-Ascoli theorem.

\begin{lemma}\label{lem:sequences}
  Fix $T>0$,  $\omega\in\mathcal{C}$, and some $s^0\in\mathcal{S}^1$. Suppose that
  \[ \left\|S^n(\omega,0)-s^0\right\|_w\leq \tilde{R}_n, \]
  for some sequence $\tilde{R}_n \downarrow 0$. Then, there exist sequences $\left\{R_n^{(i)} \downarrow 0\right\}_{i=0}^\infty$ and $\gamma_n \downarrow 0$ such that
  \[ S_i^n(\omega,\cdot) \in E_n\left(R^{(i)}_n,\gamma_n\right), \quad \forall\,i\in\mathbb{Z}_+,\ \forall\ n\geq 1, \]
with the constant $L$ in the definition of $E_n$ equal to $1+\lambda$.
\end{lemma}
\begin{proof} Fix some $\omega\in\mathcal{C}$. Based on our coupled construction, each coordinate of $A^n$ (the process of cumulative arrivals) and $D^n$ (the process of cumulative departures) is non-decreasing, and can have a positive jump, of size $1/n$, only when there is an event in $\mathcal{N}_\lambda$ or $\mathcal{N}_1$, respectively. As a result, for every $i$ and $n$, we have
  \[ \left|A^n_i(\omega,a)-A^n_i(\omega,b)\right| \leq \frac{1}{n}\left|\mathcal{N}_\lambda(\omega,na)-\mathcal{N}_\lambda(\omega,nb)\right|, \quad \forall \, a,b\in [0,T], \]
  and
  \[ \left|D^n_i(\omega,a)-D^n_i(\omega,b)\right| \leq \frac{1}{n}\left|\mathcal{N}_1(\omega,na)-\mathcal{N}_1(\omega,nb)\right|, \quad \forall \, a,b\in [0,T]. \]
  Therefore,
  \begin{align*}
   \left|S^n_i(\omega,a)-S^n_i(\omega,b)\right| \leq \frac{1}{n} & \left|\mathcal{N}_\lambda(\omega,na)-\mathcal{N}_\lambda(\omega,nb)\right| \\
    & + \frac{1}{n} \left|\mathcal{N}_1(\omega,na)-\mathcal{N}_1(\omega,nb)\right|.
    \end{align*}
  Since $\omega\in\mathcal{C}$, Lemma \ref{lem:nice_set} implies that $\frac{1}{n}\mathcal{N}_\lambda(\omega,nt)$ and $\frac{1}{n}\mathcal{N}_1(\omega,nt)$ converge uniformly on $[0,T]$ to $\lambda t$ and to $t$, respectively. Thus, there exists a pair of sequences $\gamma_n^1\downarrow 0$ and $\gamma_n^2 \downarrow 0$ (which depend on $\omega$) such that for all $n\geq 1$,
  \[ \frac{1}{n}|\mathcal{N}_\lambda(\omega,na)-\mathcal{N}_\lambda(\omega,nb)| \leq \lambda|a-b|+\gamma_n^1, \]
  and
  \[ \frac{1}{n}|\mathcal{N}_1(\omega,na)-\mathcal{N}_1(\omega,nb)| \leq |a-b|+\gamma_n^2, \]
  which imply that
  \[ \left|S^n_i(\omega,a)-S^n_i(\omega,b)\right| \leq (1+\lambda)|a-b|+(\gamma_n^1+\gamma_n^2).
 \]
  The proof is completed by setting $R^{(i)}_n=2^i\tilde{R}_n$, $\gamma_n=\gamma_n^1+\gamma_n^2$, and $L=1+\lambda$.
\end{proof}

We are now ready to prove the existence of convergent subsequences of the process of interest.

\begin{proposition}\label{prop:tightness}
  Fix $T>0$, $\omega\in\mathcal{C}$, and some $s^0\in\mathcal{S}^1$. Suppose (as in Lemma \ref{lem:sequences}) that $\|S^n(\omega,0)-s^0\|_w\leq\tilde{R}_n$, where $\tilde{R}_n\downarrow 0$. Then, every subsequence of $\left\{S^n(\omega,\cdot)\right\}_{n=1}^{\infty}$ contains a further subsequence $\left\{S^{n_k}(\omega,\cdot)\right\}_{k=1}^{\infty}$ that converges to a coordinate-wise Lipschitz continuous function $s(t)$ with $s(0)=s^0$ and
  \[ |s_i(a)-s_i(b)| \leq L|a-b|, \quad\quad \forall \, a,b\in[0,T],i\in\mathbb{Z}, \]
  where $L$ is independent of $T$, $\omega$, and $s(\cdot)$.
\end{proposition}
\begin{proof}
As in Lemma \ref{lem:sequences}, let $L=1+\lambda$.
A standard argument, similar to the one in \cite{bramson98} and
\cite{powerOfLittle}, based on the sequential compactness of $E_c$ and the ``closeness'' of $E_n\big(R^{(i)}_n,\gamma_n\big)$ to $E_c$
establishes the following. For any $i\geq 1$,  every subsequence of
 $\left\{S^n_i(\omega,\cdot)\right\}_{n=1}^{\infty}$ contains a further subsequence that converges to a Lipschitz continuous function $y_i(t)$ with $y_i(0)=s^0_i$.

Starting with the existence of coordinate-wise limit points, we now argue the existence of a limit point of $S^n$ in $D^\infty[0,T]$. Let $s_1$ be a Lipschitz continuous limit point of $\left\{S^n_1(\omega,\cdot)\right\}_{n=1}^\infty$, so that there is a subsequence such that
\[ \lim\limits_{k\to\infty} d\left(S_1^{n^1_k}(\omega,\cdot),s_1\right)=0.\]
We then proceed inductively and let
 $s_{i+1}$ be a limit point of a subsequence of $\big\{S_{i+1}^{n^i_k}(\omega,\cdot)\big\}_{k=1}^\infty$, where $\left\{n^i_k\right\}_{k=1}^\infty$ are the indices of the subsequence of $S_i^n$.

We now argue that $s$ is indeed a limit point of $S^n$ in $D^\infty[0,T]$. Fix a positive integer $i$. Because of the construction of $s$,
$S_j^{n_k^i}(\omega,\cdot)$ converges to $s_j$, as $k\to\infty$, for $j=1,\ldots,i$. In particular, there exists some $n^i>i$, for which
$$d\big(S_j^{n^i}(\omega,\cdot), s_j\big)\leq \frac{1}{i},\qquad j=1,\ldots,i.$$
We then have
\begin{align*}
     d^{\mathbb{Z}_+}\left( S^{n^i}(\omega,\cdot), s\right) &= \sup\limits_{t\in[0,T]} \sqrt{\sum\limits_{j=1}^\infty 2^{-j} \left| S_j^{n^i}(\omega,t) - s_j(t) \right|^2 } \\
     &\leq \frac{1}{i} + \sqrt{\sum\limits_{j=n^i+1}^\infty 2^{-j+2} }.
   \end{align*}
We now let $i$ increase to infinity (in which case {$n^i$} also increases to infinity), and
we conclude that $d^{\mathbb{Z}_+}\left( S^{n^i}(\omega,\cdot), s\right)\to 0$.

\end{proof}

This concludes the proof of the tightness of the sample paths. It remains to prove that any possible limit point is a fluid solution.

\subsection{Derivatives of the fluid limits}

\begin{proposition}\label{prop:derivatives}
  Fix $\omega\in\mathcal{C}$ and $T>0$. Let $s$ be a limit point of some subsequence of $\left\{S^n(\omega,\cdot)\right\}_{n=1}^\infty$.
As long as $\omega$ does not belong to a certain zero-measure subset of $\mathcal{C}$, $s$ satisfies the differential equations that define a fluid solution (cf. Definition \ref{def:fluid_model}).
\end{proposition}
\begin{proof}
We fix some $\omega\in\mathcal{C}$ and for the rest of this proof we suppress the dependence on $\omega$ in our notation. Let $\left\{S^{n_k}\right\}_{k=1}^\infty$ be a subsequence that converges to $s$, i.e.,
\[ \lim\limits_{k\to\infty} \sup\limits_{0\leq t\leq T}\left\|S^{n_k}(t)-s(t)\right\|_w=0.\]
After possibly restricting, if necessary, to a further subsequence, we can define Lipschitz continuous functions $a_i(t)$ and $d_i(t)$ as the limits of the subsequences of cumulative arrivals and departures processes $\{A_i^{n_k}(t)\}_{k=1}^\infty$ and $\{D_i^{n_k}(t)\}_{k=1}^\infty$ respectively.
Because of the relation $S_i^n(t)=S^{(0,n)}+A_i^n(t)-D_i^n(t)$, it is enough to prove the following relations, for almost all $t$:
\begin{align*}
    \frac{da_1}{dt}(t)=&\lambda[1-P_0(s(t))]+\lambda[1-s_1(t)]P_0(s(t)), \\
    \frac{da_i}{dt}(t)=&\lambda[s_{i-1}(t)-s_i(t)]P_0(s(t)), \quad\quad \forall \, i\geq 2, \\
    \frac{dd_i}{dt}(t)=&s_i(t)-s_{i+1}(t), \quad\quad \forall \, i\geq 1.
  \end{align*}
We will provide a proof only for the first one, as the other proofs are similar.
The main idea in the argument that follows is to replace the token process $M^n$ by simpler, time-homogeneous birth-death processes that are easy to analyze.

Let us fix some time $t\in(0,T)$, which is a regular time for both $a_1$ and $d_1$. Let $\epsilon>0$ be small enough so that $t+\epsilon \leq T$ and so that it also satisfies a condition to be introduced later. Equation \eqref{eq:cum_arrivals} yields
\begin{align}
  A^{n_k}_1(t+\epsilon)-A^{n_k}_1(t)=&\frac{1}{n_k} \sum\limits_{j=\mathcal{N}_\lambda({n_k}t)+1}^{\mathcal{N}_\lambda({n_k}(t+\epsilon))} \left[ \mathds{1}_{[1,c(n_k)]}\left(M^{n_k}\left({t_j^{n_k}}^-\right)\right) \right. \nonumber \\
   & \left.+ \mathds{1}_{\{0\}}\left(M^{n_k}\left({t_j^{n_k}}^-\right)\right) \mathds{1}_{\left[0,1-S_1^{n_k}\left({t_j^{n_k}}^-\right)\right)}(U(j)) \right].
\end{align}

By Lemma \ref{lem:sequences}, there exists a sequence $\gamma_{n_k} \downarrow 0$ and a constant $L$ such that
 \begin{equation*}
 S_1^{n_k}(u) \in \big[s_1(t)-\left(\epsilon L+\gamma_{n_k}\right),\ s_1(t)+\left(\epsilon L+\gamma_{n_k}\right)\big), \quad \forall\,u\in[t,t+\epsilon].
\end{equation*}
Then, for all sufficiently large $k$, we have
 \begin{equation}
 S_1^{n_k}(u) \in \big[s_{{1}}(t)-2\epsilon L,\ s_{{1}}(t)+2\epsilon L)\big), \quad  \forall\,u\in[t,t+\epsilon].
 \label{eq:set_inclusion}
\end{equation}
In particular, for $k$ sufficiently large and for every event time ${t_j^{n_k}}^-\in(t,t+\epsilon]$ of the driving process $\mathcal{N}_\lambda$, we have
\begin{align*}
 \left[0,1-S_1^{n_k}\left({t_j^{n_k}}^-\right)\right) &\subset  \Big[0,1-s_1(t)+2\epsilon L\Big).
 \end{align*}
This implies that
\begin{align*}
  A^{n_k}_1(t+\epsilon)-A^{n_k}_1(t)\leq & \frac{1}{n_k} \sum\limits_{j=\mathcal{N}_\lambda({n_k}t)+1}^{\mathcal{N}_\lambda({n_k}(t+\epsilon))}
  \Big[
  \mathds{1}_{[1,c(n_k)]}\left(M^{n_k}\left({t_j^{n_k}}^-\right)\right) \\
  &+ \mathds{1}_{\{0\}}\left(M^{n_k}\left({t_j^{n_k}}^-\right)\right) \mathds{1}_{\left[0,1-s_1(t)+2\epsilon L\right)}(U(j))
  \Big].
\end{align*}

We wish to analyze this upper bound on  $A^{n_k}_1(t+\epsilon)-A^{n_k}_1(t)$, which will then lead to an upper bound on $(da_i/dt)(t)$.
Towards this purpose, we will focus on the empirical distribution of
$\mathds{1}_{\{0\}}\left(M^{n_k}\left({t_j^{n_k}}^-\right)\right)$, which depends on the birth-death process $M^{n_k}(t)$, and which is in turn modulated by $S^{n_k}(t)$. In particular, we will define two coupled time-homogeneous birth-death processes: $M^{n_k}_+$, which is dominated by $M^{n_k}$; and $M^{n_k}_-$, which dominates $M^{n_k}$ over $(t,t+\epsilon]$, i.e.,
\begin{equation}
M^{n_k}_+(u)\leq M^{n_k}(u)\leq M^{n_k}_-(u), \quad\quad \forall\,u\in(t,t+\epsilon]. \label{eq:stoch_dominance}
\end{equation}
This is accomplished as follows. Using again Equation \eqref{eq:set_inclusion}, when $n_k$ is sufficiently large, we get the set inclusion
\[ \left[0,1-S_1^{n_k}\left({t_j^{n_k}}^-\right)-\frac{M^{n_k}\left({t_j^{n_k}}^-\right)}{{n_k}}\right) \subset \Big[0,1-s_1(t)+2\epsilon L\Big), \]
for all event times $t_j^{n_k}\in [t,t+\epsilon)$.
Furthermore, our assumptions on $c(n_k)$ imply that $M^{n_k}(t)/{n_k}\leq c(n_k)/n_k$ goes to zero as $k\to\infty$. Thus, when $n_k$ is sufficiently large,
\[ \left[0,1-S_1^{n_k}\left({t_j^{n_k}}^-\right)-\frac{M^{n_k}\left({t_j^{n_k}}^-\right)}{{n_k}}\right) \supset \Big[0,1-s_1(t)-3\epsilon L\Big), \]
for all event times $t_j^{n_k}\in [t,t+\epsilon)$. We now define intermediate coupled processes $\tilde{M}^{n_k}_+$ and $\tilde{M}^{n_k}_-$ by replacing the last indicator set in the evolution equation for $M^n(t)$ (cf.~Equation \eqref{eq:memory_def}), by the deterministic sets introduced above. Furthermore, we set $\tilde{M}^{n_k}_+(t)=0\leq M^{n_k}(t)$ and $\tilde{M}^{n_k}_-(t)=c(n_k)\geq M^{n_k}(t)$.

More concretely, for all $u\in[t,t+\epsilon]$, we let
\begin{align*}
  &\tilde{M}^{n_k}_-(u)\triangleq c(n_k) - \sum\limits_{j=\mathcal{N}_\lambda({n_k}t)+1}^{\mathcal{N}_\lambda({n_k}u)} \mathds{1}_{[1,c(n_k)]}\left(\tilde{M}^{n_k}_-\left({t_j^{n_k}}^-\right)\right) \\
  &+\sum\limits_{j=\mathcal{N}_{\mu(n_k)}({n_k}t)+1}^{\mathcal{N}_{\mu(n_k)}({n_k}u)} \mathds{1}_{[0,c(n_k)-1]}\left(\tilde{M}^{n_k}_-\left({t_j^{n_k}}^-\right)\right) \mathds{1}_{\left[0,1-s_1(t)+2\epsilon L\right)}(V(j))
\end{align*}
and
\begin{align*}
  &\tilde{M}^{n_k}_+(u)\triangleq 0 - \sum\limits_{j=\mathcal{N}_\lambda({n_k}t)+1}^{\mathcal{N}_\lambda({n_k}u)} \mathds{1}_{[1,c(n_k)]}\left(\tilde{M}^{n_k}_+\left({t_j^{n_k}}^-\right)\right) \\
  &+\sum\limits_{j=\mathcal{N}_{\mu(n_k)}({n_k}t)+1}^{\mathcal{N}_{\mu(n_k)}({n_k}u)} \mathds{1}_{[0,c(n_k)-1]}\left(\tilde{M}^{n_k}_+\left({t_j^{n_k}}^-\right)\right) \mathds{1}_{\left[0,1-s_1(t)-3\epsilon L\right)}(V(j)).
\end{align*}
We note that the processes $\tilde{M}^{n_k}_-(u)$ and $\tilde{M}^{n_k}_+(u)$ are plain, time-homogenous birth-death Markov processes, no longer modulated by $S^{n_k}(t)$, and therefore easy to analyze. It can now be argued, by induction on the event times, that $\tilde{M}^{n_k}_-(u) \geq M^{n_k}(u)$ for all $u$. We omit the details but simply note that (i) this inequality holds at time $t$; (ii) whenever the process  $M^{n_k}(u)$ has an upward jump, the same is true for $\tilde{M}^{n_k}_-(u)$, unless $\tilde{M}^{n_k}_-(u)$ is already at its largest possible value, $c(n_k)$, in which case the desired inequality is preserved; (iii) as long as the desired inequality holds, whenever the process $\tilde{M}^{n_k}_-(u)$ has a downward jump, the same is true for $M^{n_k}(u)$, unless $M^{n_k}(u)$ is already at its smallest possible value, $0$, in which case the desired inequality is again preserved. Using also a symmetrical argument for  $\tilde{M}^{n_k}_+(u)$, we obtain the domination relationship
\begin{equation}\label{eq:stoch_dominance2}
 \tilde{M}^{n_k}_+(u)\leq M^{n_k}(u)\leq \tilde{M}^{n_k}_-(u), \quad \forall\,u\in(t,t+\epsilon].
\end{equation}
Even though $\tilde{M}^{n_k}_+$ and $\tilde{M}^{n_k}_-$ are simple birth-death processes, it is convenient to simplify them even further.
We thus proceed to define the coupled processes ${M}^{n_k}_+$ and ${M}^{n_k}_-$ by modifying the intermediate processes $\tilde{M}^{n_k}_+$ and $\tilde{M}^{n_k}_-$ in a different way for each regime.
\begin{itemize}
\item [(i)] {\bf High Memory regime}: Recall that in this regime we have $\mu(n_k)=\mu$ for all $k$.
Let us fix some $l$, independently from $k$, and let
$c_l=c(n_l)$. For every $k$,
we define $M^{n_k}_+$ and $M^{n_k}_-$ by replacing the upper bound $c(n_k)$ on the number of tokens in $\tilde{M}^{n_k}_+$ and $\tilde{M}^{n_k}_-$, by $c_l$ and $\infty$ respectively. More concretely, for $u\in [t,t+\epsilon]$ we let
\begin{align*}
  &M^{n_k}_-(u)\triangleq c(n_k) - \sum\limits_{j=\mathcal{N}_\lambda({n_k}t)+1}^{\mathcal{N}_\lambda({n_k}u)} \mathds{1}_{[1,\infty)}\left(M^{n_k}_-\left({t_j^{n_k}}^-\right)\right) \\
  & +\sum\limits_{j=\mathcal{N}_\mu({n_k}t)+1}^{\mathcal{N}_\mu({n_k}u)}  \mathds{1}_{\left[0,1-s_1(t)+2\epsilon L\right)}(V(j))
\end{align*}
and
\begin{align*}
  &M^{n_k}_+(u)\triangleq 0 - \sum\limits_{j=\mathcal{N}_\lambda({n_k}t)+1}^{\mathcal{N}_\lambda({n_k}u)} \mathds{1}_{[1,c_l]}\left(M^{n_k}_+\left({t_j^{n_k}}^-\right)\right) \\
  &+\sum\limits_{j=\mathcal{N}_\mu({n_k}t)+1}^{\mathcal{N}_\mu({n_k}u)} \mathds{1}_{[0,c_l-1]}\left(M^{n_k}_+\left({t_j^{n_k}}^-\right)\right) \mathds{1}_{\left[0,1-s_1(t)-3\epsilon L\right)}(V(j)).
\end{align*}
When $k$ is large enough, we have $c(n_k)\geq c_l$, and as we are replacing $c(n_k)$ by $c_l$ in $\tilde{M}^{n_k}_+$, we are reducing the state space of the homogeneous birth-death process $\tilde{M}^{n_k}_+$. It is easily checked (by induction on the events of the processes) that we have the stochastic dominance $\tilde{M}^{n_k}_+ \geq M^{n_k}_+$. Using a similar argument, we obtain $\tilde{M}^{n_k}_- \leq M^{n_k}_-$. These facts, together with Equation \eqref{eq:stoch_dominance2}, imply the desired dominance relation in Equation \eqref{eq:stoch_dominance}.

\item [(ii)] {\bf High Message regime}: Recall that in this regime we have $c(n_k)=c$, for all $k$. Let us fix some $l$, independently from $k$, and let $\mu_l=\mu(n_l)$. We define $M^{n_k}_+$ by replacing the process $\mathcal{N}_{\mu(n_k)}$ that generates the spontaneous messages  in $\tilde{M}^{n_k}_+$, by $\mathcal{N}_{\mu_l}$. More concretely, for $u\in [t,t+\epsilon]$ we let
\begin{align*}
  &M^{n_k}_+(u)\triangleq 0 - \sum\limits_{j=\mathcal{N}_\lambda({n_k}t)+1}^{\mathcal{N}_\lambda({n_k}u)} \mathds{1}_{[1,c]}\left(M^{n_k}_+\left({t_j^{n_k}}^-\right)\right) \\
  &+ \sum\limits_{j=\mathcal{N}_{\mu_l}\left({n_k}t\right)+1}^{\mathcal{N}_{\mu_l} \left({n_k}u\right)} \mathds{1}_{[0,c-1]}\left(M^{n_k}_+\left({t_j^{n_k}}^-\right)\right) \mathds{1}_{\left[0,1-s_1(t)-3\epsilon L \right)}(V(j)).
\end{align*}
Recall that we assumed that the event times in the Poisson process $\mathcal{N}_{\mu(n_k)}$ are a subset of the event times of $\mathcal{N}_{\mu(n_{k+1})}$, for all $k$. As a result, when $k\geq l$, the process $M^{n_k}_+$ only has a subset of the upward jumps in $\tilde{M}^{n_k}_+$, and thus (using again a simple inductive argument) satisfies $\tilde{M}^{n_k}_+ \geq M^{n_k}_+$. Furthermore, we define $M^{n_k}_-(u)\triangleq c$, which clearly satisfies $\tilde{M}^{n_k}_- \leq M^{n_k}_-$. Combining these facts with Equation \eqref{eq:stoch_dominance2}, we have again the desired dominance relation in Equation \eqref{eq:stoch_dominance}.

\item [(iii)] {\bf Constrained regime}: Recall that in this regime we have $c(n_k)=c$ and $\mu(n_k)=\mu$, for all $k\geq 1$. For this case, we define $M^{n_k}_-=\tilde{M}^{n_k}_-$ and $M^{n_k}_+=\tilde{M}^{n_k}_+$, which already satisfy the desired dominance relation in Equation \eqref{eq:stoch_dominance}.\end{itemize}

For all three regimes, and having fixed $l$, the dominance relation in
Equation \eqref{eq:stoch_dominance} implies that when $k$ is large enough ($k\geq l$), we have
\[ \mathds{1}_{\{0\}}\left(M_-^{n_k}\left({t_j^{n_k}}^-\right)\right) \leq \mathds{1}_{\{0\}}\left(M^{n_k}\left({t_j^{n_k}}^-\right)\right) \leq \mathds{1}_{\{0\}}\left(M_+^{n_k}\left({t_j^{n_k}}^-\right)\right) \]
for all ${t_j^{n_k}}^- \in (t,t+\epsilon]$. Consequently,
\begin{align}
 A^{n_k}_1(t+\epsilon)-A^{n_k}_1(t)\leq & \frac{1}{n_k} \sum\limits_{j=\mathcal{N}_\lambda({n_k}t)+1}^{\mathcal{N}_\lambda({n_k}(t+\epsilon))} \Big[1-\mathds{1}_{\{0\}}\left(M_-^{n_k}\left({t_j^{n_k}}^-\right)\right) \nonumber \\
  &+ \mathds{1}_{\{0\}}\left(M_+^{n_k}\left({t_j^{n_k}}^-\right)\right) \mathds{1}_{\left[0,1-s_1(t)+2\epsilon L\right)}(U(j))
  \Big]. \label{eq:fluid_ineq}
\end{align}
Note that the transition rates of the birth-death processes $M_-^{n_k}$ and $M_+^{n_k}$, for different $n_k$, involve $n_k$ only as a scaling factor. As a consequence, the corresponding steady-state distributions are the same for all $n_k$.

Let $P_0^-(s(t))$ and $P_0^+(s(t))$ be the steady-state probabilities of  state $0$ for $M_-^{n_k}$ and $M_+^{n_k}$, respectively.
Then, using the PASTA property, we have that as $n_k\to\infty$, the empirical averages
\begin{equation}\label{eq:emp}
  \frac{1}{n_k} \sum\limits_{j=\mathcal{N}_\lambda({n_k}t)+1}^{\mathcal{N}_\lambda({n_k}(t+\epsilon))} \mathds{1}_{\{0\}}\left(M_-^{n_k}\left({t_j^{n_k}}^-\right)\right)
 \end{equation}
and
  \[ \frac{1}{n_k} \sum\limits_{j=\mathcal{N}_\lambda({n_k}t)+1}^{\mathcal{N}_\lambda({n_k}(t+\epsilon))} \mathds{1}_{\{0\}}\left(M_+^{n_k}\left({t_j^{n_k}}^-\right)\right) \]
converge almost surely to  $\epsilon\lambda P_0^-(s(t))$ and $\epsilon\lambda P_0^+(s(t))$, respectively.

We now continue with the explicit calculation of
$P_0^-(s(t))$ and $P_0^+(s(t))$.

\begin{itemize}
  \item [(i)] {\bf High Memory regime}:
  \[ P_0^-(s(t)) = \left[ 1-\frac{\mu\cdot\min\{1-s_1(t)+ 2\epsilon L,1\}}{\lambda} \right]^+ , \]
 and
 \[ P_0^+(s(t))= \left[ \sum\limits_{k=0}^{c_l} \left( \frac{\mu \big(1-s_1(t)-3 \epsilon L\big)^+}{\lambda} \right)^k \right]^{-1}, \]
  \item [(ii)] {\bf High Message regime}: If $s_1(t)<1$, then we assume that $\epsilon$ has been chosen small enough so that $1-s_1(t)-3\epsilon L >0$. We then obtain
     \[ P_0^-(s(t))= 0 \]
  and
  \[ P_0^+(s(t))= \left[ \sum\limits_{k=0}^{c} \left( \frac{\mu_l[1-s_1(t)-3\epsilon L]^+}{\lambda} \right)^k \right]^{-1}. \]

Suppose now that $s_1(t)=1$. In this case, the approach based on the processes $M_-^{n_k}$ and $M_+^{n_k}$ is not useful, because it yields $P_0^-(s(t))= 0$ and $P_0^+(s(t))=1$, for all $\epsilon>0$ and for all $\mu_l$. This case will be considered separately later.

  \item [(iii)] {\bf Constrained regime}:
   \[ P_0^-(s(t))= \left[ \sum\limits_{k=0}^{c} \left( \frac{\mu\cdot\min\{1-s_1(t)+2\epsilon L,1\}}{\lambda} \right)^k \right]^{-1} \]
 and
 \[ P_0^+(s(t))= \left[ \sum\limits_{k=0}^{c} \left( \frac{\mu \big(1-s_1(t)-3\epsilon L\big)^+}{\lambda} \right)^k \right]^{-1}, \]
\end{itemize}

We now continue by considering all three regimes, with the exception of the High Message regime with $s_1(t)=1$, which will be dealt with separately. We use the fact that the random variables $U(j)$ are independent from the process $M_+^{n_k}$.
Using an elementary argument, which is omitted, it can be seen that
  \[ \frac{1}{n_k} \sum\limits_{j=\mathcal{N}_\lambda({n_k}t)+1}^{\mathcal{N}_\lambda({n_k}(t+\epsilon))} \mathds{1}_{\{0\}}\left(M_+^{n_k}\left({t_j^{n_k}}^-\right)\right) \mathds{1}_{\left[0,1-s_1(t)+2\epsilon L\right)}(U(j)) \]
converges to the limit of the  empirical average in Equation \eqref{eq:emp}, which is the product of
$\epsilon\lambda P_0^+(s(t))$ times the expected value of
$\mathds{1}_{\left[0,1-s_1(t)+2\epsilon L\right)}(U(j))$. That is, it converges to $\epsilon P_0^+(s(t))\min\{1-s_1(t)+2\epsilon L,1\}$, $\mathbb{P}$-almost surely.

Recall that we have fixed some $\epsilon>0$ and some $l$ and, furthermore, that $P_0^-$ and $P_0^+$ depend on $l$ for the High Memory and High Message regimes, and on $\epsilon$ for all regimes. We will first take limits, as $k\to\infty$, while holding $\epsilon$ and $l$ fixed.
Using the inequality in Equation \eqref{eq:fluid_ineq}, and the fact that the left-hand side converges to the fluid limit $a(t+\epsilon)-a(t)$ as $k\to\infty$, we obtain
 \[ a_1(t+\epsilon)-a_1(t)\leq \epsilon\lambda[1-P_0^-(s(t))]+\epsilon\lambda P_0^+(s(t))\min\{1-s_1(t)+2\epsilon L,1\}. \]
 An analogous argument yields
 \[ a_1(t+\epsilon)-a_1(t)\geq \epsilon\lambda[1-P_0^+(s(t))]+\epsilon\lambda P_0^-(s(t))[1-s_1(t)-2\epsilon L]^+. \]
We now take the limit as $l\to\infty$, so that $c_l\to\infty$ for the
High Memory regime and $\mu_l\to\infty$ for the High Message regime, and then take the limit as $\epsilon\to 0$. Some elementary algebra shows that
in all cases, $P_0^+(s(t))$ and $P_0^-(s(t))$ both converge to $P_0(s(t))$, as defined in the statement of the proposition. We thus obtain
\begin{equation}\label{eq:a1Derivative}
 \frac{da_1(t)}{dt}=\lambda[1-P_0(s(t))]+\lambda[1-s_1(t)]P_0(s(t)),
\end{equation}
as desired.

We now return to the exceptional case of the High Message regime with $s_1(t)=1$, and find the derivative of $a_1(t)$ using a different argument. Recall that we have the hard bound $S^n_1(t)\leq 1$, for all $t$ and for all $n$. This leads to the same bound for the fluid solutions, i.e., $s_1(t)\leq 1$ for all $t$. As a result, since  $t>0$ is a regular time, we must have $\dot{s}_1(t)=0$. Furthermore, we also have the formula
\begin{equation*}\label{eq:dotD}
  \dot{d}_1(t)=s_1(t)-s_2(t)=1-s_2(t),
\end{equation*}
which is established by an independent argument, using the same proof technique as for $\dot{a}_1$, but without the inconvenience of having to deal with $M^{n_k}$. Then, since $\dot{s}_1(t)=\dot{a}_1(t)-\dot{d}_1(t)$, we must also have
\begin{equation}\label{eq:Adot}
\dot{a}_1(t)=1-s_2(t).
\end{equation}
On the other hand, it can be easily checked that $\dot{a}_1(t)\leq \lambda$ for all regular $t$, and thus we must have $s_2(t)\geq 1-\lambda$. We have thus established that at all regular times $t>0$ with $s_1(t)=1$, $s_2(t)$ must be at least $1-\lambda$. Then it follows (cf. Definition \ref{def:fluid_model}) that at time $t$, we have
\[ P_0(s(t)) = \left[1-\frac{1-s_2(t)}{\lambda}\right]^+ \mathds{1}_{\{s_1(t)=1\}}=1-\frac{1-s_2(t)}{\lambda}. \]
It is then easily checked that Equation \eqref{eq:Adot} is of the form
\[ \dot a_1(t)=\lambda(1-P_0(s(t)))+\lambda(1-s_1(t))P_0(s(t)), \]
exactly as in Equation \eqref{eq:a1Derivative}, where the last equality used the property $s_1(t)=1$.


The derivatives of $a_i$, for $i>1$, and of $d_i$, for $ i\geq 1$, are obtained using similar arguments, which are omitted.
\end{proof}

For every sample path outside a zero-measure set, we have established the following. Proposition \ref{prop:tightness} implies the existence of limit points of the process $S^n$. Furthermore, according to Proposition \ref{prop:derivatives} these limit points verify the differential equations of the fluid model. Since all stochastic trajectories $S^n(t)$ take values in $\mathcal{S}$ (which is a closed set), their limits are functions taking values in $\mathcal{S}$ as well. We will now show that the limit $s(t)$ actually belongs to the smaller set $\mathcal{S}^1$, which is a requirement in our definition of fluid solutions. Using the same argument as in the proof of Proposition \ref{prop:asymptitic_stability}, it can be shown that
\[ \frac{d}{dt}\|s(t)\|_1 \leq \lambda, \]
for all regular times $t$. Since the trajectories $s$ are continuous with respect to our weighted norm $\|\cdot\|_w$, but not necessarily with respect to the $1$-norm, it now remains to be checked that the $1$-norm cannot become infinite at a nonregular time.

Suppose that $t_1$ is a nonregular time. Recall, from the proof of Proposition \ref{prop:asymptitic_stability}, that such a time may occur only once, and only in the High Message regime, if trajectory hits the set
\[ D = \{s\in\mathcal{S}:s_1=1,\,\, s_2>1-\lambda\}. \]
For all $t<t_1$, we have $P_0(s(t))=0$, and thus $\dot s_i(t)\leq 0$, for all $t<t_1$ and all $i\geq 2$. Combining this with the continuity of the coordinates, we obtain $s_i(t_1)\leq s_i(0)$, for all $i\geq 2$. It follows that
\[ \|s(t_1)\|_1 \leq 1+ s_1(t_1)+\sum\limits_{i=2}^\infty s_i(0) \leq 2+\|s(0)\|_1. \]
Combining this with the fact that $\|s(0)\|_1<\infty$, we get that $\|s(t)\|_1<\infty$, for all $t\geq 0$, and thus $s(t)\in\mathcal{S}^1$, for all $t\geq 0$. This implies the existence of fluid solutions, thus completing the proof of Theorem  \ref{thm:equilibrium}.

Moreover, we have already established a uniqueness result in Theorem \ref{thm:equilibrium}: for any initial condition $s^0\in\mathcal{S}^1$, we have at most one fluid solution. We also have (Proposition \ref{prop:tightness}) that every subsequence of $S^n$ has a further subsequence that converges --- by necessity to the same (unique) fluid solution. It can be seen that this implies the convergence of $S^n$ to the fluid solution, thus proving Theorem \ref{thm:fluid_limit}.

\section{{Stochastic steady-state analysis --- Proofs} of Proposition \ref{prop:ergodicity} and Theorem \ref{prop:interchange}} \label{sec:proof_inter}

In this section, we prove Proposition \ref{prop:ergodicity} and Theorem \ref{prop:interchange}, which assert that for any finite $n$, the stochastic system  is positive recurrent with some invariant distribution $\pi^n$ and that the sequence of the marginals of the invariant distributions, $\left\{\pi^n_s\right\}_{n=1}^{\infty}$, converges in distribution to a measure concentrated on the unique equilibrium of the fluid model. These results guarantee that the properties derived from the equilibrium $s^*$ of the fluid model, and specifically for the asymptotic delay, are an accurate approximation of the steady state of the stochastic system for $n$ large enough.

\subsection{Stochastic stability of the $n$-th system} \label{sec:proof_stoch_stab}

We will use the Foster-Lyapunov criterion to show that for any fixed $n$, the continuous-time Markov process $\left(S^n(t),M^n(t)\right)$ is positive recurrent.

Our argument is developed by first considering a detailed description of the system:
\[\left({Q}_1(t),\dots,{Q}_n(t),{M}^n(t)\right),\]
which keeps track of the size of each queue, but without keeping track of the identities of the servers with associated tokens in the virtual queue. As hinted in Section \ref{s:ssr}, this is a continuous-time Markov process, on the state space
\[ Z_n \triangleq
\left\{ (q_1,\dots,q_n,m)\in \mathbb{Z}_+^n\times\{0,1,\dots,c(n)\}: \sum\limits_{i=1}^{n} \mathds{1}_{\{q_i=0\}} \geq m \right\}. \]
The transition rates, denoted by $r^n_{\cdot\,\to\,\cdot}$  are as follows, where we use $e_i$ to denote the $i$-th unit vector in $\mathbb{Z}^n_+$.
\begin{enumerate}
  \item When there are no tokens available $(m=0)$, each queue sees arrivals with rate $\lambda$:
  \[ r^n_{(q,0)\to\left(q+e_i,0\right)}= \lambda, \quad \quad  i=1,\dots,n. \]
  \item When there are tokens available $(m>0)$, the arrival stream, which has rate $n\lambda$, is divided equally between all empty queues:
  \[ r^n_{(q,m)\to \left(q+e_i,m-1\right)} = \frac{n\lambda\mathds{1}_{\{q_i=0\}}}{\sum\limits_{j=1}^{n} \mathds{1}_{\{q_j=0\}}}\mathds{1}_{\{m>0\}}, \quad  i=1,\dots,n. \]
  \item Transitions due to service completions occur at a uniform rate of $1$ at each queue, and they do not affect the token queue:
  \[ r^n_{(q,m)\to \left(q-e_i,m\right)} = \mathds{1}_{\{q_i>0\}}, \quad \quad  i=1,\dots,n. \]
  \item Messages from idling servers are sent to the dispatcher (hence resulting in additional tokens) at a rate equal to $\mu(n)$ times the number of empty servers that do not already have associated tokens in the virtual queue:
      \begin{equation*}
         r^n_{(q,m)\to (q,m+1)} = \mu(n) \left(\sum\limits_{i=1}^{n} \mathds{1}_{\{q_i=0\}} -m\right)\mathds{1}_{\{m<c(n)\}}.
      \end{equation*}
\end{enumerate}
Note that the Markov process of interest, $\left(S^n(t),M^n(t)\right)$, is a function of the process $\big({Q}(t),{M}^n(t)\big)$. Therefore, to establish positive recurrence of the former, it suffices to establish positive recurrence of the latter, as in the proof that follows.

\begin{proof}[Proof of Proposition \ref{prop:ergodicity}]
  The Markov process $\big({Q}(t),{M}^n(t)\big)$ on the state space $Z_n$ is clearly irreducible, with all states reachable from each other. To show positive recurrence, we define the quadratic Lyapunov function
  \begin{equation}\label{eq:lyapunov}
    \Phi(q,m)\triangleq \frac{1}{n} \sum\limits_{i=1}^{n} q_i^2,
  \end{equation}
  and note that
  \[ \sum\limits_{(q',m')\neq(q,m)} \Phi(q',m')r^n_{(q,m)\to(q',m')} < \infty,  \quad\quad \forall \, (q,m)\in Z_n.  \]
  We also define the finite set
  \[ F_n\triangleq \left\{(q,m)\in Z_n : \sum\limits_{i=1}^{n} q_i < \frac{n(\lambda+2)}{2(1-\lambda)} \right\}. \]
  As $q_i$ can change but at most $1$ during a transition, we use the relations $(q_i+1)^2-q_i^2=2q_i+1$ and $(q_i-1)^2-q_i^2=-2q_i+1$.
For any $(q,m)$ outside the set $F_n$, we have
  \begin{align*}
    &\sum\limits_{(q',m')\in Z_n}  \left[\Phi(q',m')-\Phi(q,m)\right]r^n_{(q,m)\to(q',m')} \\
    &= \frac{1}{n} \sum\limits_{i=1}^{n} \left[ (2q_i+1)\lambda \left(  \frac{n\mathds{1}_{\{q_i=0\}}}{\sum\limits_{j=1}^{n} \mathds{1}_{\{q_j=0\}}} \mathds{1}_{\{m>0\}} +\mathds{1}_{\{m=0\}} \right) -(2q_i-1)\mathds{1}_{\{q_i>0\}} \right] \\
    &= \lambda + \frac{1}{n} \sum\limits_{i=1}^{n}  \left[ \mathds{1}_{\{q_i>0\}} -2q_i\left(1-\lambda\mathds{1}_{\{m=0\}}\right)\right]
      \end{align*}
  \begin{align*}
&\leq \lambda + 1 - \frac{2(1-\lambda)}{n} \sum\limits_{i=1}^{n}  q_i \leq -1, \quad \forall \,(q,m)\,{\in Z_n\backslash F_n}.
  \end{align*}
The last equality is obtained through a careful rearrangement of terms; the first inequality is obtained by replacing each indicator function by unity.
Then, the Foster-Lyapunov criterion \cite{fosterLyapunov} implies positive recurrence.
\end{proof}

\subsection{Convergence of the invariant distributions} \label{sec:proof_interchange}
As a first step towards establishing the interchange of limits result, we start by establishing some tightness properties, in the form of uniform (over all $n$) upper bounds for $\mathbb{E}_{\pi^n}\left[ \|S^n\|_1 \right]$ and for $\pi^n\left( Q_1^n \geq k \right)$. One possible approach to obtaining such bounds is to use an appropriate coupling and show that our system is stochastically dominated by a system consisting of $n$ independent parallel M/M/1 queues. However, we follow an easier approach based on a simple linear Lyapunov function and the results in \cite{hajekLyapunov} and \cite{gamarnikBounds}.

\begin{lemma}\label{lem:bound}
Let $\pi^n$ be the unique invariant distribution of the process $\left(Q^n(t),M^n(t)\right)$. We then have the uniform upper bounds
   \[ \pi^n\left( Q_1^n \geq k \right) \leq \left(\frac{1}{2-\lambda}\right)^{k/2}, \quad \forall \, n, \ \forall \, k, \]
and
  \[ \mathbb{E}_{\pi^n}\left[ \|S^n\|_1 \right]\leq 2+\frac{2}{1-\lambda}, \quad\quad \forall\, n. \]
\end{lemma}
\begin{proof}
  Consider the linear Lyapunov function
  \[ \Psi(q,m)=q_1. \]
Under the terminology in \cite{gamarnikBounds}, this Lyapunov function has exception parameter $B=1$, drift $\gamma=1-\lambda$, maximum jump $\nu_{\rm max}=1$, and maximum rate $p_{\rm max}\leq 1$. Note that this function is not a witness of stability because the set $\left\{(q,m)\in Z_n:\Psi(q,m)<1\right\}$ is not finite. However, the boundedness of the upward jumps allows us to use Theorem 2.3 from \cite{hajekLyapunov} to obtain that $\mathbb{E}_{\pi^n}\left[ Q_1^{n} \right]<\infty$. Thus, all conditions in Theorem 1 in \cite{gamarnikBounds} are satisfied, yielding the upper bounds
  \[ \pi^n\left( Q_1^n \geq 1+2m \right) \leq \left(\frac{1}{2-\lambda}\right)^{m+1},\qquad \forall\ m\geq 1, \]
and
  \[ \mathbb{E}_{\pi^n}\left[ Q_1^n \right]\leq 1+\frac{2}{1-\lambda}. \]
The first part of the result is obtained by letting $m=(k-1)/2$ if $k$ is odd or $m=k/2-1$ if $k$ is even.
  Finally, using the definition $\|S^n\|_1=1+\frac{1}{n} \sum\limits_{i=1}^n  Q_i$, which, together with symmetry yields
  \begin{align*}
     \mathbb{E}\left[ \|S^n\|_1 \right]  &= 1+ \frac{1}{n} \sum\limits_{i=1}^n  \mathbb{E}\left[Q_i\right] = \mathbb{E}\left[Q_1\right],
  \end{align*}
  and concludes the proof.
\end{proof}

We now prove our final result on the interchange of limits.

\begin{proof}[Proof of Theorem \ref{prop:interchange}]
Consider the set $\mathbb{Z}_+ \cup \{\infty\}$ endowed with the topology of the Alexandroff compactification, which is known to be metrizable. Moreover, it can be seen that the topology defined by the norm $\|\cdot\|_w$ on $[0,1]^{\mathbb{Z}_+}$ is equivalent to the product topology, which makes $[0,1]^{\mathbb{Z}_+}$ compact. As a result, the product $\{s\in\mathcal{S}^1: \|s\|_1\leq M\} \times ( \mathbb{Z}_+\cup\{\infty\})$ is closed, and thus compact, for all $M$. Note that, for each $n$, the invariant distribution $\pi^n$ is defined over the set $(\mathcal{S}^1\cap \mathcal{I}_n)\times\{0,1,\dots,c(n)\}$. This is a subset of $\mathcal{S}^1 \times ( \mathbb{Z}_+\cup\{\infty\})$, so we can extend the measures $\pi^n$ to the latter, larger set.

Let $\{S^n(0)\}_{n=1}^\infty$ be a sequence of random variables distributed according to the marginals $\{\pi^n_s\}_{n=1}^\infty$. From Lemma \ref{lem:bound}, we have
\begin{equation}\label{eq:tightBound}
 \mathbb{E}_{\pi^{n}}\big[ \|S^n(0)\|_1 \big]\leq 2+\frac{2}{1-\lambda}, \quad \forall\, n.
\end{equation}
Using Markov's inequality, it follows that for every $\epsilon>0$, there exists a constant $M$ such that
\begin{equation*}
\pi^n_s\big(\{s\in\mathcal{S}^1: \|s\|_1\leq M\}\big)\geq 1- \epsilon, \quad \forall \, n,
\end{equation*}
which implies that
\begin{equation*}
\pi^n\big(\{s\in\mathcal{S}^1: \|s\|_1\leq M\} \times ( \mathbb{Z}_+\cup\{\infty\})\big)\geq 1- \epsilon, \quad \forall \, n.
\end{equation*}
Thus, the sequence $\{\pi^n\}_{n=1}^\infty$ is tight and, by Prohorov's theorem \cite{billingsley}, it is also relatively compact in the weak topology on the set of probability measures. It follows that any subsequence has a weakly convergent subsequence whose limit is a probability measure over $\mathcal{S}^1 \times ( \mathbb{Z}_+\cup\{\infty\})$.

Let $\left\{\pi^{n_k}\right\}_{k=1}^\infty$ be a weakly convergent subsequence, and let $\pi$ be its limit. Let $S(0)$ be a random variable distributed according to $\pi_s$, where $\pi_s$ is the marginal of $\pi$. Since $\mathcal{S}^1 \times ( \mathbb{Z}_+\cup\{\infty\})$ is separable, we invoke Skorokhod's representation theorem to construct a probability space $(\Omega_0,\mathcal{A}_0,\mathbb{P}_0)$ and a sequence of random variables $\left(S^{n_k}(0),M^{n_k}(0)\right)$ distributed according to $\pi^{n_k}$, such that
  \begin{equation}
  \lim\limits_{k\to\infty} \left\|S^{n_k}(0)- S(0)\right\|_w= 0 \quad\quad \mathbb{P}_0-a.s. \label{eq:almost_sure}
  \end{equation}
We use the random variables $\left(S^{n_k}(0),M^{n_k}(0)\right)$ as the initial conditions for a sequence of processes $\left\{\left(S^{n_k}(t),M^{n_k}(t)\right)\right\}_{k=1}^\infty$, so that each one of these processes is stationary. Note that the initial conditions, distributed as $\pi^{n_k}$, do not necessarily converge to a deterministic initial condition (this is actually part of what we are trying to prove), so we cannot use Theorem \ref{thm:fluid_limit} directly to find the limit of the sequence of processes $\left\{S^{n_k}(t)\right\}_{k=1}^\infty$. However, given any $\omega\in\Omega_0$ outside a zero $\mathbb{P}_0$-measure set, we can restrict this sequence of stochastic processes to the probability space
  \[ \left(\Omega_\omega,\mathcal{A}_\omega,\mathbb{P}_\omega\right) = \left(\Omega_D\times\Omega_S\times\{\omega\}, \mathcal{A}_D\times\mathcal{A}_S\times\{\emptyset,\{\omega\}\}, \mathbb{P}_D\times\mathbb{P}_S\times\delta_{\omega}\right) \]
  and apply Theorem \ref{thm:fluid_limit} to this new space, to obtain
  \[ \lim\limits_{k\to\infty} \sup\limits_{0\leq t\leq T} \left\|S^{n_k}(t,\omega)-S(t,\omega)\right\|_w =0, \quad \mathbb{P}_\omega-a.s., \]
  where $S(t,\omega)$ is the fluid solution with initial condition $S(0,\omega)$. Since this is true for all $\omega\in\Omega_0$ except for a set of zero $P_0$-measure, it follows that
  \[ \lim\limits_{k\to\infty} \sup\limits_{0\leq t\leq T} \left\|S^{n_k}(t)-S(t)\right\|_w = 0, \quad\quad \mathbb{P}-a.s., \]
  where $\mathbb{P}=\mathbb{P}_D\times\mathbb{P}_S\times\mathbb{P}_0$
  and where $S(t)$ is another stochastic process whose randomness is only in the initial condition $S(0)$ (its trajectory is the deterministic fluid solution for that specific initial condition).

We use Lemma \ref{lem:bound} once again to interchange limit, expectation, and infinite summation in Equation \eqref{eq:tightBound} (using the same argument as in Lemma \ref{lem:delay}) to obtain
\[ \mathbb{E}_{\pi_s} \big[ \|S(0)\|_1 \big] \leq 2+\frac{2}{1-\lambda}. \]
Using Markov's inequality now in the limit, it follows that for every $\epsilon>0$, there exists a constant $M$ such that
\begin{equation}
\pi_s\big(\|S(0)\|_1\leq  M\big)\geq 1- \epsilon. \label{eq:tight}
\end{equation}

Recall that the uniqueness of fluid solutions (Theorem \ref{thm:equilibrium}) implies the continuous dependence of solutions on initial conditions \cite{filippov}. Moreover, Theorem \ref{thm:equilibrium} implies that any solution $s(t)$ with initial conditions $s(0) \in \mathcal{S}^1$ converges to $s^*$ as $t\to\infty$. As a result, there exists $T_\epsilon>0$ such that
 \[ \sup\limits_{s(0):\, \|s(0)\|_1\leq M}  \left\|s(T_\epsilon)-s^*\right\|_w < \epsilon. \]
Combining this with Equation \eqref{eq:tight}, we obtain
\begin{align}
  \mathbb{E}_{\pi_s}\big[\|S(T_\epsilon)-s^*&\|_w\big] \nonumber \\
    &= \mathbb{E}_{\pi_s}\Big[ \|S(T_\epsilon)-s^*\|_w \ \Big| \ \|S(0)\|_1\leq M \Big] \, \pi_s\big(\|S(0)\|_1\leq M \big) \nonumber \\
&\quad + \mathbb{E}_{\pi_s}\Big[ \|S(T_\epsilon)-s^*\|_w \ \Big| \ \|S(0)\|_1> M \Big] \, \pi_s\big( \|S(0)\|_1> M \big) \nonumber \\
&< \epsilon + \big( \sup\limits_{s\in\mathcal{S}} \|s-s^*\|_w \big) \epsilon \nonumber \\
&\leq 2\epsilon, \label{eq:expectations}
\end{align}
where the expectations $\mathbb{E}_{\pi_s}$ are with respect to the random variable $S(0)$, distributed according to $\pi_s$, even though the dependence on $S(0)$ is suppressed from our notation and is left implicit. On the other hand, due to the stationarity of $S^{n_k}(\cdot)$, the random variables $S^{n_k}(0)$ and $S^{n_k}(T_\epsilon)$ have the same distribution, for any $k$. Taking the limit as $k\to\infty$, we see that $S(0)$ and $S(T_\epsilon)$ have the same distribution. Combining this with Equation \eqref{eq:expectations}, we obtain
\[ \mathbb{E}_{\pi_s}\big[\|S(0)-s^*\|_w\big] \leq 2\epsilon. \]
Since $\epsilon$ was arbitrary, it follows that $S(0)=s^*$, $\pi_s$-almost surely, i.e., the distribution $\pi_s$ of $S(0)$ is concentrated on $s^*$. We have shown that the limit $\pi_s$ of a convergent subsequence of $\pi^n$ is the Dirac measure $\delta_{s^*}$. Since this is true for every convergent subsequence and $\pi^n$ is tight, this implies that $\pi^n$ converges to $\delta_{s^*}$, as claimed.
\end{proof}

\section{Conclusions and future work}\label{sec:conclusions}
The main objective of this paper was to study the tradeoff between the amount of resources (messages and memory) available to a central dispatcher, and the expected queueing delay as the system size increases. We introduced a parametric family of pull-based dispatching policies and, using a fluid model and associated convergence theorems, we showed that with enough resources, we can drive the queueing delay to zero as the system size increases.

We also analyzed a resource constrained regime of our pull-based policies that, although it does not have vanishing delay, it has some remarkable properties. We showed that by wisely exploiting an arbitrarily small message rate (but still proportional to the arrival rate) we obtain a queueing delay which is finite and uniformly upper bounded for all $\lambda<1$, a significant qualitative improvement over the delay of the M/M/1 queue (obtained when we use no messages). Furthermore, we compared it with the popular power-of-$d$-choices and found that while using the same number of messages, our policy achieves a much lower expected queueing delay, especially when $\lambda$ is close to $1$.

Moreover, in a companion paper we show that \emph{every} dispatching policy (within a broad class of policies) that uses the same amount of resources as our policy in the constrained regime, results in a non-vanishing queueing delay. This implies that our family of policies is able to achieve vanishing delay with the minimum amount of resources in some sense.

There are several interesting directions for future research.
\begin{itemize}
\item [(i)] It would be interesting to extend these results to the case of different service disciplines such as processor sharing or LIFO, or to the case of general service time distributions, as these are prevalent in many applications.
\item [(ii)] We have focused on a system with homogeneous servers. For the case of nonhomogeneous servers, even stability can become an issue, and there are interesting tradeoffs between the resources used and the stability region.
\item [(iii)] Another interesting line of work is to consider a reverse problem, in which we have decentralized arrivals to several queues, a central server, and a scheduler that needs to decide which queue to serve. In this context we expect to find again a similar tradeoff between the resources used and the queueing delay.
\end{itemize}

\appendix

\section{Interchange of limit, expectation, and infinite summation} \label{app:delay}

\begin{lemma}\label{lem:delay}
We have
\[ \lim_{n\to\infty} \mathbb{E}\left[ \sum\limits_{i=1}^{\infty} S_i^n \right] = \sum\limits_{i=1}^{\infty} s_i^*. \]
\end{lemma}
\begin{proof}
By Fubini's theorem, we have
\[ \lim_{n\to\infty} \mathbb{E}\left[ \sum\limits_{i=1}^{\infty} S_i^n \right] = \lim_{n\to\infty}  \sum\limits_{i=1}^{\infty} \mathbb{E}\left[ S_i^n \right]. \]
Due to the symmetric nature of the invariant distribution $\pi^n$, we have
\begin{align*}
  \mathbb{E}\left[ S_i^n \right] &= \mathbb{E}\left[ \frac{1}{n} \sum_{j=1}^n \mathds{1}_{[i,\infty)}\left(Q_j^n\right) \right] \\
&= \mathbb{E}\left[ \mathds{1}_{[i,\infty)}\left(Q_1^n\right) \right] \\
&= \pi^n\left( Q_1^n \geq i \right) \\
&\leq \left( \frac{1}{2-\lambda} \right)^{{i}/{2}},
\end{align*}
where the inequality is established in Lemma \ref{lem:bound}. We can therefore apply the Dominated Convergence Theorem to interchange the limit with the first summation, and obtain
\[  \lim_{n\to\infty} \sum\limits_{i=1}^{\infty} \mathbb{E}\left[ S_i^n \right] =  \sum\limits_{i=1}^{\infty}  \lim_{n\to\infty} \mathbb{E}\left[ S_i^n \right],\]
We already know that $S^n_i$ converges to $s^*$, in distribution (Theorem \ref{prop:interchange}). Then, using a variant of the Dominated Convergence Theorem for convergence in distribution, and the fact that we always have $S^n_i\leq 1$, we can finally interchange the limit and the expectation and obtain
\begin{align*}
\sum\limits_{i=1}^{\infty}  \lim_{n\to\infty} \mathbb{E}\left[ S_i^n \right] &= \sum\limits_{i=1}^{\infty} s_i^*.
\end{align*}
\end{proof}



\bibliographystyle{imsart-number}
\bibliography{references}
\end{document}